\documentclass[a4,11pt]{amsart}

\usepackage[mathscr]{eucal}
\usepackage{amsmath}
\usepackage{amsthm}
\usepackage{amssymb}
\usepackage{amsfonts} % f"urs \mathbb Alphabet 
\usepackage{mathrsfs}
\usepackage{amscd}
\usepackage{enumerate} %um Numerierungsz"ahler hochz"ahlen zu k"onnen
\usepackage{color,wrapfig} % f"ur Bilder
\usepackage[pdftex]{graphicx}
\usepackage{layout}
\usepackage{float} % fuer Positionierung von figures
\usepackage{mathtools} % for /abs
\newtheorem{theorem}[equation]{Theorem}
\newtheorem*{IntroTheorem}{Theorem}

\newtheorem{lemma}[equation]{Lemma}
\newtheorem{definition}[equation]{Definition}
\newtheorem{remark}[equation]{Remark}

\newtheorem{definition and proposition}[equation]{Definition and Proposition}
\def\epsilon{\varepsilon}
\def\phi{\varphi}

\newcommand{\al}{\alpha}
\newcommand{\be}{\beta}
\newcommand{\ga}{\gamma}
\newcommand{\de}{\delta}
\newcommand{\ep}{\epsilon}

\newcommand{\si}{\sigma}

\newcommand{\om}{\omega}

\newcommand{\gadot}{{\dot{\ga}}}

\newcommand{\phidot}{{\dot{\phi}}}

\newcommand{\xiti}{{\ti{\xi}}}

\newcommand{\fti}{{\ti{f}}}

\newcommand{\xti}{{\ti{x}}}
\newcommand{\yti}{{\ti{y}}}
\newcommand{\zti}{{\ti{z}}}

\newcommand{\xbar}{\bar{x}}
\newcommand{\ybar}{\bar{y}}
\newcommand{\zbar}{\bar{z}}

\newcommand{\Fbar}{\bar{F}}

\def\N{{\mathbb N}}

\def\R{{\mathbb R}}

\newcommand{\mcF}{\mathcal F}
\newcommand{\mcG}{\mathcal G}

\newcommand{\mcM}{\mathcal M}

\newcommand{\mcV}{\mathcal V}
\newcommand{\mcW}{\mathcal W}
\newcommand{\mcX}{\mathcal X}
\newcommand{\mcY}{\mathcal Y}
\newcommand{\mcZ}{\mathcal Z}

\newcommand{\mcMhat}{\widehat{\mathcal M}}

\newcommand{\mfx}{\mathfrak x}

\newcommand{\mfxti}{\ti{{\mathfrak x}}}

\newcommand{\mfmhat}{{\hat{\mathfrak m}}}

\newcommand{\ti}{\tilde}
\newcommand{\x}{\times}
\newcommand{\del}{\partial}

\newcommand{\beq}{\begin{equation}}
\newcommand{\eeq}{\end{equation}}
\newcommand{\beqs}{\begin{equation*}}
\newcommand{\eeqs}{\end{equation*}}

\DeclareMathOperator{\Crit}{Crit}

\DeclareMathOperator{\grad}{grad}

\DeclareMathOperator{\Ind}{Ind}

\def\slashii#1{\setbox0=\hbox{$#1$}             % set a box for #1 
\dimen0=\wd0                                 % and get its size
\setbox1=\hbox{\sl/} \dimen1=\wd1            % get size of /
\ifdim\dimen0>\dimen1                        % #1 is bigger
\rlap{\hbox to \dimen0{\hfil\sl/\hfil}}   % so center / in box
#1                                        % and print #1
\else                                        % / is bigger
\rlap{\hbox to \dimen1{\hfil$#1$\hfil}}   % so center #1
\hbox{\sl/}                               % and print /
\fi}                                         %
\def\slashiii#1{\setbox0=\hbox{$#1$}#1\hskip-\wd0\hbox to\wd0{\hss\sl/\/\hss}}
\newcommand{\refgluingassociative}{Theorem \ref{gluingassociative}}

\newcommand{\refmodulispacecompact}{Theorem \ref{modulispacecompact}}

\newcommand{\refmodulicorners}{Subsection \ref{modulicorners}}

\newcommand{\refrelativeMorseIndex}{Remark \ref{relativeMorseIndex}}

\newcommand{\refmorseOnmanifoldsWithCorners}{Theorem \ref{morseOnManifoldsWithCorners}}

\newcommand{\refnglobular}{Definition \ref{nglobular}}

\newcommand{\refncategory}{Definition \ref{ncategory}}

\newcommand{\refmorsencategory}{Theorem \ref{morsencategory}}

\begin{document}

% Title of document, usually lower case except for first word
% and proper nouns.  Avoid unnecessary symbols.

\title{Higher Morse moduli spaces and $n$-categories}

% If the title is too long for the running head, use
% the following command to specify a short title:
%\shorttitle{Shorter title}

% First Author
\author{Sonja Hohloch}             

% First Author's email address.  Must come before \address.
\email{sonja.hohloch@epfl.ch}

% First author's postal address, with *line breaks* like below.  Do
% not use \\ to separate lines.  It will appear on one line in the
% article, but will be used exactly as typed below to send a
% complimentary copy of the journal, so ensure that it is a complete
% postal address with line breaks as would appear on an envelope.
%
% Also, if it is not obvious, please add a comment with % that says
% what part is a district within a city, what part is a city, 
% what part is a region or province and what part is a postal code.
%
% We recommend using the most current address possible.  
% If you want the journal copy sent to a different address than
% the one below, please let Dan Christensen <jdc@uwo.ca> know.
\address{Section de Math\'ematiques, % department
         \'Ecole Polytechnique F\'ed\'erale de Lausanne (EPFL), % university
         SB MATHGEOM CAG, Station 8, % code of post office of math department
         1015 Lausanne, % zip code + city
         Switzerland}

% If needed, use a \thanks command, but not inside the author command.
\thanks{The author was partially supported by the DFG grant Ho 4394/1-1.}

% AMS 2010 Mathematics Subject Classification.  List one or several,
% separated by commas, ending in a period.
% See http://www.ams.org/mathscinet/msc/msc2010.html for the 2010
% numbering system.
%\classification{18B99, 18D99, 55U99, 58E05.}
% Use \classification[2000]{12X34, 55X78.} if you must use codes
% from the 2000 numbering system.

% Keywords of the article, usually singular, no leading caps.  
% Separated by commas, ending with period.
%\keywords{Morse theory, $n$-category theory, flow category.}

% Abstract comes before maketitle
\begin{abstract}
We generalize Cohen $\&$ Jones $\&$ Segal's flow category whose objects are the critical points of a Morse function and whose morphisms are the Morse moduli spaces between the critical points to an $n$-category.
The $n$-category construction involves repeatedly doing Morse theory on Morse moduli spaces for which we have to construct a class of suitable  Morse functions. It turns out to be an `almost strict' $n$-category, i.e. it is a strict $n$-category `up to canonical isomorphisms'.
%We compute the new $n$-category structure on the $n$-sphere, the deformed $2$-sphere and the $2$-torus.

\vspace{2mm}

\noindent
MSC2010 classification: 18B99, 18D99, 37D15, 57R99, 58E05.

\noindent
Keywords: Morse theory, $n$-category theory, flow category.
\end{abstract}

\maketitle

% Text of Document.  Use constructs such as \section, \subsection,
% \begin{theorem} ... \end{theorem}, \begin{proof} ... \end{proof}, etc.

%%%%%%%%%%%%%%%%%%%%%%%%%%%%%%%%%%%%%%%%%%%%%%%%%%%%%%%%%%%%%%%%%%%%%%%%%%%%%%%%%%%%%%%%%%%%%%%%%%%%%%%%%%%%55
%%%%%%%%%%%%%%%%%%  new section  %%%%%%%%%%%%%%%%%%%%%%%%%%%%%%%%%%%%%%%%%%%%%%%%%%%%%%%%%%%%%%%%%%%%%%%%%%%%%
%%%%%%%%%%%%%%%%%%%%%%%%%%%%%%%%%%%%%%%%%%%%%%%%%%%%%%%%%%%%%%%%%%%%%%%%%%%%%%%%%%%%%%%%%%%%%%%%%%%%%%%%%%%%%%

\section{Introduction}

The aim of the present paper is to study $n$-category structures in Morse theory. There are (at least) two good reasons to do so: On the one hand, there is an astonishing example and on the other hand there is the natural generalization of a category defined by Cohen $\&$ Jones $\&$ Segal \cite{cohen-jones-segal}.

\vspace{2mm}

Let us first have a look at the seminal example. Let $n \in \N_0$ and $0 \leq k \leq n$ and consider the $(n-k)$-sphere
\beqs
\mathbb S^{n-k}=\{(x_1,\dots, x_n) \in \R^n \mid x^2_1 + \dots + x_{n-k}^2=1, x_{n-k+1} = \dots =x_n=0 \}
\eeqs
with the standard metric.
Denote by $f_{n-k}\colon \mathbb S^{n-k} \to \R$, $f_{n-k}(x_1, \dots, x_{n-k}):=x_{n-k}$ the height function of $\mathbb S^{n-k}$ and by $N_{n-k}$ and $S_{n-k}$ its north and south pole which are the only critical points of $f_{n-k}$. Now consider the compactified Morse moduli space of unparametrized negative gradient flow trajectories between the north and the south pole denoted by $\mcMhat(N_{n-k}, S_{n-k}):= \overline{\mcM(N_{n-k}, S_{n-k}, f_{n-k}) \slash \R}$. For all $0 \leq k \leq n$, we observe $\mcMhat(N_{n-k}, S_{n-k}) \simeq \mathbb S^{n-k} \cap \{x_{n-k} = \dots =x_n=0\} \simeq \mathbb S^{n-(k+1)}$. 

\vspace{2mm}

This means that we can in fact iterate: Start with $k=0$ and consider $\mathbb S^n$ with the Morse function $f_n$. The Morse moduli space $\mcMhat(N_n, S_n)$ is isomorphic to $\mathbb S^{n-1}$ which is a nice manifold. Thus we can consider $k=1$ and choose $f_{n-1}$ as Morse function on $\mcMhat(N_n, S_n) \simeq \mathbb S^{n-1}$ and obtain the new Morse moduli space $\mcMhat(N_{n-1}, S_{n-1}) \simeq \mathbb S^{n-2}$. This game can be repeated until $k=n$. Roughly, we get a `filtration' of Morse moduli spaces with one lying `in' or `on' the other.

\vspace{2mm}

A natural question now is: What happens if we start with an arbitrary manifold instead of a sphere?
So let $M$ be a smooth, closed, $n$-dimensional manifold and $(f,g)$ a Morse-Smale pair on $M$, i.e.\ $f$ is a Morse function, $g$ a Riemannian metric and the (un)stable manifolds intersect each other transversely. The first question is: How do the arising Morse moduli spaces look like? Are they nice enough spaces to admit Morse theory? The literature tells us (cf.\ \refmodulispacecompact) that, under slight assumptions on the metric, the moduli spaces are manifolds with corners, i.e.\ manifolds modeled on $(\R_{\geq 0})^n$. On manifolds with boundary (possibly with corners), one can do Morse theory as has been shown by Braess \cite{braess}, Goresky $\&$ MacPherson \cite{goresky-macpherson}, Akaho \cite{akaho}, Kronheimer $\&$ Mrowka \cite{kronheimer-mrowka}, Ludwig \cite{ludwig}, Handron \cite{handron} and Laudenbach \cite{laudenbach}. 
On manifolds with boundary, there are two Morse theory approaches possible: 
\begin{enumerate}[(1)]
 \item 
Morse functions whose gradient vector field is {\em transverse} to the boundary. 
\item
Morse functions which induce a gradient vector field {\em tangent} to the boundary. 
\end{enumerate}
We will use the second option since it is nicely compatible with lower dimensional boundary strata and admits a complete gradient flow. {\em Thus all gradient vector fields in this paper are always tangent to the boundary}.

\vspace{2mm}

Now let $x$, $y \in \Crit(f)$ and consider the moduli space $\mcMhat(x,y,f)$ which is a manifold with corners. On $\mcMhat(x,y,f)$ we want to choose a Morse function whose gradient vector field is tangent to the boundary. More precisely, instead of `boundary' we should rather speak of {\em boundary strata}. The literature tells us (cf.\ \refmodulispacecompact) how the strata look like: The boundary is a union of products of (lower dimensional) Morse moduli spaces. This requires a `compatibility condition' for the Morse function in case different moduli spaces `share' certain strata. Moreover, for reasons which become later apparent, we want the Morse functions to decrease from higher dimensional strata to lower dimensional strata. Such Morse functions can be recursively constructed. Now choose such a Morse function on $\mcMhat(x,y,f)$ with a suitable metric, call the Morse function $\fti$ and consider its moduli spaces $\mcMhat(\xti, \yti, \fti)$ for $\xti$, $\yti \in \Crit(\fti)$. Once again, if we want to 
iterate, we have to ask ourselves: What kind of space is $\mcMhat(\xti, \yti, \fti)$? Can we do Morse theory on it? The construction of $\fti$, in particular the fact that $\fti$ is decreasing from higher to lower dimensional strata, enables us to prove that $\mcMhat(\xti, \yti, \fti)$ is again a manifold with corners (cf. \refmorseOnmanifoldsWithCorners). And this holds true for the Morse moduli spaces of a similar constructed Morse function $\ti{\fti}$ on $\mcMhat(\xti, \yti, \fti)$ such that we can continue to consider Morse moduli spaces on Morse moduli spaces. 

We consider the compactified {\em unparametrised} moduli spaces such that the dimension decreases at least by one in comparison to the dimension of the space on which we are working. This means that our iteration of Morse moduli spaces on Morse moduli spaces becomes trivial after at most $\dim M +1$ steps.

\vspace{2mm}

Before we investigate how the above iteration of moduli spaces gives rise to an $n$-category structure, let us have a look at the second motivation for this paper.

\vspace{2mm}

In the 1990's, Cohen $\&$ Jones $\&$ Segal \cite{cohen-jones-segal} came up with the following category: Let $M$ be a smooth closed manifold with a Morse-Smale pair $(f,g)$. Then the objects $Obj(\mcF)$ of the {\em flow category} $\mcF$ are given by the critical points of $f$, i.e.\ $Obj(\mcF)=\Crit(f)$, and the morphisms between two objects are given by the compactified Morse trajectory spaces, i.e.\ $Morph(x,y):=\mcMhat(x,y):= \overline{\mcM(x,y,f)\slash \R}$ for $x$, $y \in \Crit(f)=Obj(\mcF)$. 
According to Cohen $\&$ Jones $\&$ Segal \cite{cohen-jones-segal}, the classifying space $B\mcF$ of $\mcF$ is homeomorphic to $M$. If the gradient flow is not Morse-Smale the classifying space $B\mcF$ is only homotopy equivalent to $M$. These are certainly intriguing observations, but we are interested in $\mcF$ for a different reason: what is happening if we try to `iterate' this category? With that we mean to introduce a `second level' where the objects are given by the above defined morphisms and the new morphisms are `morphisms between morphisms'. 
More generally, define the objects of the $k$th level to be the morphisms of the $(k-1)$th level and the morphisms of the $k$th level to be the morphisms between the morphisms of the $(k-1)$th level. This iteration procedure is inspired by a paper by Baez \cite{baez} where it is used to motivate the notion of $n$-categories.

\vspace{2mm}

To the best of our knowledge, it is still not clear what the best definition of an $n$-category is. In the literature, there is a whole zoo of various definitions what $n$-categories are supposed to be. Leinster's book \cite{leinster} gives a good introduction to this topic.

On the first glance, $n$-category theory distinguishes between `weak' and `strict' $n$-categories. This distinction is (among others) related to the question if the composition of morphisms is `really' associative or only associative `up to some degree', i.e. if for the composition of three morphisms $\phi$, $\psi$, $\chi$ holds $(\phi \circ \psi )\circ \chi= \phi \circ (\psi \circ \chi) $ or only $(\phi\circ \psi )\circ \chi \sim \phi \circ (\psi \circ \chi) $ where $\sim$ may stand for instance for `homotopy equivalent' or something else.

It is known that `weak' and `strict' $n$-categories are equivalent for $n \in \{0,1,2\}$, but already for $n=3$ (and higher $n$) these notions differ. In the following, we will stick to the conventions of Leinster's book \cite{leinster}. The definition of a strict $n$-category is lengthy such that we will not line it out here in the introduction, but we refer the reader to \refnglobular\ and \refncategory\ or directly to Leinster's book. Since we will not work with weak $n$-categories we also refer the reader to Leinster's book for their definition.

\vspace{2mm}

In this paper, we will show that the iteration procedure of Morse moduli spaces as sketched above will give rise to an {\em almost strict} $n$-category. With {\em almost strict} we mean that our structure satisfies the conditions of a strict $n$-category {\em up to canonical isomorphisms}. Since the `up to whatever'-defect of a weak category is usually much bigger than just `up to canonical isomorphisms' we opt for calling the structure `almost strict' instead of `not very weak'. This may be up for discussion, but from a geometer's point of view it makes sense. For a geometer, for example, the cartesian product is associative whereas in fact one probably would have to say `associative up to canonical isomorphism'. Since the whole construction of compactified, unparametrized moduli spaces involves already taking equivalence classes etc.\ we would get nowhere if we would not admit `up to canonical isomorphism' to be negligible. Our contructions certainly do not need `too large' deformations.

\begin{IntroTheorem}
The above described iteration of Morse moduli spaces can be given the structure of an almost strict $n$-category. The resulting $n$-category is denoted by $\mcX$.
\end{IntroTheorem}

Let us summarize briefly our almost strict $n$-category. The intuitive `level structure' will be replaced by a so-called $n$-globular set (see \refnglobular) with source and target functions which `remember' on which `level' an element lives. The elements of our $n$-globular set are tupels of a moduli space and a critical point on this moduli space. The identity functions make use of the stationary moduli spaces $\mcMhat(x,x)$. And the composition is based on the gluing of Morse trajectories.

Recall that when we sketched the generalization of the n-sphere example, we required the Morse function to decrease from higher to lower dimensional strata. This is not just a technical assumption. If we admit an arbitrary Morse function, the Morse moduli spaces, more precisely their boundaries, become more complicated. And in particular we do not obtain an $n$-category structure, but rather some `opetopes' (cf. Hohloch $\&$ Ludwig \cite{hohloch-ludwig}).

\vspace{2mm}

So far, each in this way constructed almost strict $n$-category depends on a certain number of {\em chosen} Morse data. In order to show independence of the chosen data in Morse or Floer theory, often a so-called `homotopy of homotopy' argument is used (cf. Schwarz \cite{schwarz} or Salamon \cite{salamon}). The `homotopy of homotopies' induces a chain homomorphism on the chain complexes associated to the chosen Morse or Floer data. This homomorphism relates the chain complexes similarly as a so-called {\em distributor} (cf.\ Borceux \cite{borceux}) works in category theory. We hope to pursue this idea in a future work in order to relate such Morse $n$-categories to each other, hopefully obtaining an $\om$-category of $n$-categories. 

\vspace{2mm}

Up to the author's knowledge, the present work is the first one to deal with higher categories associated to Morse theory (with `higher' we mean $n \geq 3$). In symplectic geometry and knot theory, $2$-categories appear for example via the Wehrheim-Woodward category (see Wehrheim $\&$ Woodward \cite{wehrheim-woodward} and Weinstein \cite{weinstein}) or the $2$-category in Khovanov homology (see Khovanov \cite{khovanov}).

\vspace{2mm}

To understand the $n$-category of Morse moduli spaces better, the subsequent article Hohloch \cite{hohloch} defines two other $n$-categories $\mcV$ and $\mcW$ and $n$-category functors $\mcF \colon \mcX \to \mcV$ and $\mcG \colon \mcX \to \mcW$. Both categories are motivated by the Morse index and the dimension of the Morse moduli spaces. $\mcV$ `sees' the elements of $\mcX$ as tuples of vector spaces and homomorphisms spaces between these vector spaces whose dimension is induced by the Morse index and the dimension of the Morse moduli spaces. $\mcW$ captures only the Morse indices and thus consists of tuples of (nonnegative) integers. For the exact definitions, we refer the reader to Hohloch \cite{hohloch}. Taking and evaluating the Morse index and the dimension of the involved Morse moduli spaces leads to the functors $\mcF \colon \mcX \to \mcV$ and $\mcG \colon \mcX \to \mcW$.

\begin{IntroTheorem}[Hohloch \cite{hohloch}]
There are two other almost strict $n$-categories called $\mcV$ and $\mcW$ and $n$-category functors $\mcF \colon \mcX \to \mcV$ and $\mcG \colon \mcX \to \mcW$.
\end{IntroTheorem}

There were many new algebraic structures discovered and studied in (symplectic) geometry in the last decade which, apart from the already above mentioned reasons, got the author interested in studying new structures in Morse theory. So for instance the graded differential algebra (DGA) appeared, defined for knots and links by Chekanov \cite{chekanov}, which is a very special case of the more general Symplectic Field Theory (SFT) started by Eliashberg $\&$ Givental $\&$ Hofer \cite{eliashberg-givental-hofer}.

\subsection*{Organization of the paper}

In Section 2, we recall, introduce and construct whatever parts of Morse theory we need in the following sections. In Section 3, we recall the notion of strict $n$-categories. In Section 4, we construct our almost strict $n$-category $\mcX$ of Morse moduli spaces. Section 5 calculates some examples.

\subsection*{Acknowledgements}
First of all, the author wants to thank Gregor Noetzel for many and lengthy discussions in the beginning of this project. Moreover, the author is indebted to Eric Finster, Kathryn Hess and Michael Warren for explanations of n-category theory and to Dan Burghelea and Ursula Ludwig for discussions and useful hints concerning Morse theory. The author was partially supported by the DFG grant Ho 4394/1-1.

%%%%%%%%%%%%%%%%%%%%%%%%%%%%%%%%%%%%%%%%%%%%%%%%%%%%%%%%%%%%%%%%%%%%%%%%%%%%%%%%%%%%%%%%%%%%%%%%%%%%%%%%
%%%%%%%%%% new section  %%%%%%%%%%%%%%%%%%%%%%%%%%%%%%%%%%%%%%%%%%%%%%%%%%%%%%%%%%%%%%%%%%%%%%%%%%%%%%%
%%%%%%%%%%%%%%%%%%%%%%%%%%%%%%%%%%%%%%%%%%%%%%%%%%%%%%%%%%%%%%%%%%%%%%%%%%%%%%%%%%%%%%%%%%%%%%%%%%%%%%%%

\section{Morse moduli spaces}

%%%%%%%%%%%%%%%%%%%%%%%%%%%%%%%%%%%%%%%%%%%%%%%%%%%%%%%%%%%%%%%%%%%%%%%%%%%%%%%%%%%%%%%%%%%%%%
%%%%%%%%%%%%%  subsection

\subsection{Notations}

Let us start with recalling some notations from Morse theory. There are several approaches to Morse theory: The classical one uses level sets and attaching of handle bodies as e.g. described in Milnor's book \cite{milnor}. But there is also a dynamical approach via the gradient flow as for instance described in Schwarz' book \cite{schwarz}. We are interested in the dynamical version and summarize the setting briefly.

\vspace{2mm} 

Let $M$ be a closed manifold. A function $f \colon M \to \R$ is called a {\em Morse function} if its Hessian $D^2f$ is nondegenerate at the critical points $\Crit(f):=\{ x \in M \mid Df(x)=0\}$. 
At a critical point $x \in \Crit(f)$, this admits the definition of the {\em Morse index} $\Ind(x)$ as the number of negative eigenvalues of $D^2f(x)$.
Given a Riemannian metric $g$ on M, we denote by $\grad_{g}f$ the gradient of $f$ w.r.t. the metric $g$. This leads to the following autonomous ODE of the {\em negative gradient flow} $\phi_t$ of the pair $(f,g)$
\beqs
\phidot_t = - \grad_g f(\phi_t).
\eeqs
Given a critical point $x \in \Crit(f)$, we define the {\em stable manifold}
\beqs
W^s(f,x):= W^s(f,g,x):= \{ p \in M \mid  \lim_{t \to + \infty} \phi_t(p)=x\}
\eeqs
and the {\em unstable manifold}
\beqs
W^u(f,x):= W^u(f,g,x):= \{ p \in M \mid \lim_{t \to - \infty} \phi_t(p) =x\}.
\eeqs
A pair $(f,g)$ is called {\em Morse-Smale} if $W^s(f,g,x)$ and $W^u(f,g,y)$ intersect transversely for all $x$, $y \in \Crit(f)$. 
We define the {\em Morse moduli space} between two critical points $x$ and $y$ as the space
\beqs
\mcM(x,y):= \mcM(x,y, f,g):= 
\left\{
 \ga\colon \R \to M \left|
\begin{aligned}
 & \gadot(t) = - \grad_g f(\ga(t)), \\ 
 & \lim_{t \to - \infty} \ga(t)=x , \\ 
 & \lim_{t \to + \infty} \ga(t)=y
\end{aligned}
\right.
\right\}.
\eeqs
$\mcM(x,y)$ consists of the negative gradient flow lines running from $x$ to $y$. It can also be identified with $W^u(x,f) \cap W^s(f,y)$. 
For a Morse-Smale pair $(f,g)$, the moduli space $\mcM(x,y)$ is a smooth manifold of dimension $\Ind(x)- \Ind(y)$. If $\Ind(y)>\Ind(x)$ then the space $\mcM(x,y)$ is empty.
Given $\ga \in \mcM(x,y)$ and $\si \in \R$, then $\ga_\si$ with $\ga_\si(t):= \ga(t+\si)$ is a gradient flow line. Thus the moduli space carries an $\R$-action via $\R \x \mcM(x,y) \to \mcM(x,y)$, $(\ga, \si) \mapsto \ga_\si$. If we divide by the action, we obtain the {\em unparametrized} moduli space $\mcM(x,y)\slash \R$. 

\vspace{2mm}

We introduce for $x$, $y \in \Crit(f)$ with $x \neq y$ the notation $x>y$ if $\mcM(x,y) \neq \emptyset$.

\vspace{2mm}

Before we continue the discussion of the Morse moduli spaces, we need some notation about manifolds with corners.
There are different notions and conventions in the literature. Manifolds with corners had been studied first by Cerf \cite{cerf} and Douady \cite{douady} at the beginning of the 1960's. An overview over the various definitions and their differences may be found in Joyce \cite{joyce}.

\vspace{2mm}

For our purposes, an $m$-dimensional manifold with corners is an $m$-dimensional manifold which is locally modeled on $\R^m_+:=(\R_{\geq 0})^m$. In order to keep track of the boundary strata, we introduce the following additional notions.
Let $N$ be an $m$-dimensional manifold with corners and let $\psi=(\psi_1, \dots, \psi_m) \colon U \subseteq N \to \R^m_+$ be a chart.
For $x \in U$ we define 
\beqs
depth(x):=\#\{i \mid \psi_i(x)=0,\ 1 \leq i \leq m \}.
\eeqs
$depth(x)$ is independent of the chosen chart. A {\em face} of $N$ is the {\em closure} of a connected component of $\{x \in N \mid depth(x)=1\}$. If $k$ is the number of faces, we fix an {\em order} of the faces and denote them by $\del_1 N, \dots, \del_k N$. The quantity $depth(x)$ counts the number of faces intersecting in $x$. 
We call the connected components of $\{x \in N \mid depth(x)=l\}=:D_{\dim N -l}$ the {\em ($\dim N -l$)-strata} of $N$.
This yields a filtration $N=\bigcup_{0 \leq j \leq \dim N}D_j$ and suggests the following definition.

\begin{definition}
Let $N$ be an $m$-dimensional manifold with corners with $k$ faces $\del_1 N, \dots, \del_k N$. We call $N$ a {\em $\langle k \rangle$-manifold} if 
\begin{enumerate}[(a)]
 \item 
Each $x \in N$ lies in $depth(x)$ faces.
\item
$\del_1 N \cup \dots \cup \del_k N = \del N$.
\item
For all $1 \leq i,j \leq k$ with $i \neq j$ the intersection $\del_i N \cap \del_j N $ is a face of both $\del_i N$ and $\del_j N$. 
\end{enumerate}
\end{definition}

In this convention, $\del_i N \subset N$ is again a manifold with corners, but $\del N$ is not. We are following Joyce's \cite{joyce} definition where the integer $\langle k \rangle	$ has a priori nothing to do with the dimension $m$ of the manifold $N$ since, in the subsequent work in progress Hohloch $\&$ Ludwig \cite{hohloch-ludwig}, it is convenient to track the intersection behaviour of each face. Note that other authors like Laures \cite{laures} are defining $\del_i N$ to be a union of faces which allows choosing $k= \dim N$.

\vspace{2mm}

The standard example of a $\langle k \rangle$-manifold is $\R^k_+$ with faces $\del_i\R^k_+:= \{ x \in \R^k_+ \mid x_i=0\}$. $\langle 0 \rangle$-manifolds are manifolds without boundary and $\langle 1 \rangle$-manifolds are manifolds with one (smooth) boundary component.

\vspace{2mm}

$\langle k \rangle$-manifolds have nice properties, see J\"anich \cite{jaenich}, Joyce \cite{joyce}, Laures \cite{laures}. For example, $\langle k \rangle$-manifolds can be embedded into an euclidean space such that the faces meet each other perpendicular (cf. so-called {\em neat embeddings} in Laures \cite{laures}). And $\langle k \rangle$-manifolds admit well-defined collar neighbourhoods of their boundaries (cf. Laures \cite{laures}).

\vspace{2mm}

Now let us return to Morse moduli spaces.
Let for instance be $x$, $y$, $z \in \Crit(f)$ with $\Ind(x) > \Ind(y) > \Ind(z)$. As sketched in Figure \ref{breaking} (b), a sequence of trajectories $(\ga_n)_{n \in \N}$ from $x$ to $z$ may `break' in the limit into trajectories $\ga_{xy}$ from $x$ to $y$ and $\ga_{yz}$ from $y$ to $z$. This phenomenon is usually denoted by `breaking' and plays an important role if one wants to compactify unparametrized Morse moduli spaces.
More precisely, an unparametrized moduli spaces can be compactified by adding `broken trajectories' and we denote the compactification of $\mcM(x,z)\slash \R$ via adding broken trajectories by $\mcMhat(x,z):= \overline{\mcM(x,z)\slash \R}$. 
In order to obtain a nice structure for the compactification one needs to pose conditions on the metric.
If $f$ is a Morse function and if a metric $g$ is euclidean near the critical points of $f$ the we call $g$ an {\em $f$-euclidean metric}.

The following statement was used often as a folklore theorem. Proofs for different settings can be found in Burghelea \cite{burghelea}, Wehrheim \cite{wehrheim} and Qin \cite{qin1}, \cite{qin2}. We summarize the statement as follows.

\begin{theorem}
\label{modulispacecompact}
Let $M$ be a closed manifold, let $(f,g)$ be Morse-Smale and assume $g$ to be $f$-euclidean. Let $x$, $z \in \Crit(f)$ with $x>z$. Then there exists $k \in \N_0$ such that $\mcMhat(x,z)$ is an $(\Ind(x)-\Ind(z)-1)$-dimensional $\langle k \rangle$-manifold with corners and its boundary is given by 
\beqs
\del \mcMhat(x,z) = \bigcup_{\stackrel{ (\Ind(x) - \Ind(z) -1) \geq l \geq 0}{x > y_1 > \dots > y_l >z}} \mcMhat(x,y_1) \x\mcMhat(y_1, y_2) \x \dots \x \mcMhat(y_{l-1}, y_l) \x \mcMhat(y_l, z)
\eeqs
where $y_1$, \dots, $y_l \in \Crit(f)$. There is a canonical smooth structure on $\mcMhat(x,z)$.
\end{theorem}

The `inverse procedure' of breaking is `gluing'. Roughly, the gluing procedure takes a broken trajectory $(\ga_{xy}, \ga_{yz}) \in \mcMhat(x,y) \x \mcMhat(y,z)$ and yields a Morse trajectory from $x$ to $z$.
If one wants to glue a multiply broken trajectory $(\ga_1, \dots, \ga_{l+1}) \in \mcMhat(x,y_1) \x \dots \x \mcMhat(y_l, z)$ the question of associativity of the gluing procedure arises. 
Gluing is indeed associative and, as a folklore theorem, it has been used a lot in the literature. 
Recently Qin \cite{qin3} and Wehrheim \cite{wehrheim} have written down proofs. For details, we refer to their works.

\begin{theorem}[\cite{qin3}, \cite{wehrheim}]
\label{gluingassociative}
Gluing is associative. 
\end{theorem}

%%%%%%%%%%%%%%%%%%%%%%%%%%%%%%%%%%%%%%%%%%%%%%%%%%%%%%%%%%%%%%%%%%%%%%%%%%%%%%%%%%%%%%%%%%%%%%%%%%%%%%
%%%%%%%%%%%%%%%%  new subsection  %%%%%%%%%%%%%%%%%%%%%%%%%%%%%%%%%%%%%%%%%%%%%%%%%%%%%%%%%%%%%%%

\subsection{Morse moduli spaces on $\langle k \rangle$-manifolds: Our construction}

\label{modulicorners}

As for manifolds without boundary, there are several ways to define Morse theory on manifolds {\em with} boundary. There is the classical approach via handle attachment by Braess \cite{braess} for manifolds with boundary and finally by Goresky $\&$ MacPherson \cite{goresky-macpherson} for stratified spaces.
And there is the newer approach via the gradient flow described by Akaho \cite{akaho} and Kronheimer $\&$ Mrowka \cite{kronheimer-mrowka} for manifolds with smooth boundary. Ludwig \cite{ludwig} finally defined Morse theory with tangential vector field on stratified spaces.

\vspace{2mm}

For our purposes, we are interested in a Morse theory where the gradient vector field is tangential to the boundary. If the gradient vector field is tangent to the boundary it is in particular tangent to all lower strata of the boundary and therefore `compatible' with corners (where it simply vanishes). If we would require the gradient vector field to be transverse to the boundary it would be much more difficult to come up with a consistent definition at the corners.

\vspace{2mm}

Working with tangential gradient vector fields is a special case of Ludwig's \cite{ludwig} setting. But since Ludwig \cite{ludwig} is interested in setting up a Morse theory, she needs only the one and two dimensional moduli spaces. We are primarily interested in very special constructions of the Morse functions and higher dimensional Morse moduli spaces which are not covered in Ludwig \cite{ludwig}.

\vspace{2mm}

This is the rough outline of the following construction: Consider a Morse function on a smooth manifold. Its compactified Morse moduli spaces are manifolds with corners. On these manifolds with corners we choose `good' Morse functions and consider their compactified Morse moduli spaces which are again manifolds with corners. On these spaces repeat the procedure etc. In each step, we loose at least one dimension since dividing by the $\R$-action in $\mcM(x,y) \slash \R$ reduces the dimension by one. Therefore the compactified moduli spaces will become zero dimensional after a finite number of iterations and the iteration becomes trivial.

\vspace{2mm}

Let $M$ be a closed manifold. 
Let $(f_0, g_0)$ be a Morse-Smale pair consisting of a Morse function $f_0$ with $f_0$-euclidean metric $g_0$. Let $x_0$, $z_0 \in \Crit(f_0)$ be distinct critical points and consider $\mcMhat(x_0, z_0, f_0)$. If this moduli space is not empty then, by \refmodulispacecompact, it is a manifold (possibly) with corners.
The boundary of $\mcMhat(x_0, z_0, f_0)$ is of the form
\beqs
\del \mcMhat(x_0,z_0,f_0) = \bigcup_{\stackrel{ (\Ind(x_0) - \Ind(z_0) -1) \geq l \geq 0}{x_0 > y^1_0 > \dots > y^l_0 >z_0}} \mcMhat(x_0,y^1_0, f_0) \x \dots \x \mcMhat(y^l_0, z_0, f_0)
\eeqs
where $y^1_0$, \dots, $y^l_0 \in \Crit(f_0)$. If we apply the formula recursively to the factors of the product we can also write
\beqs
\del \mcMhat(x_0, z_0, f_0)= \bigcup_{y_0 \in \Crit(f_0)} \mcMhat(x_0, y_0, f_0) \x \mcMhat(y_0, z_0, f_0).
\eeqs
Keep in mind that a moduli space may have several connected components. 
By labelling the components of depth one by $\del_1\mcMhat(x_0, z_0, f_0)$, \dots, $\del_k\mcMhat(x_0, z_0, f_0)$, we give $\mcMhat(x_0, z_0, f_0)$ the structure of a $\langle k \rangle $-manifold for some $k \in \N_0$. In other words, we are dealing with a stratified space. And $\mcMhat(x_0, z_0, f_0)$ might share strata with other moduli spaces $\mcMhat(\xti_0, \zti_0, f_0)$ for $\xti_0$, $\zti_0 \in \Crit(f_0)$.

\vspace{2mm}

Now we want to define a Morse function $f_1$ and a metric $g_1$ on $\mcMhat(x_0, z_0, f_0)$ for all $x_0$, $z_0$ which is compatible with the stratification and the `sharing' of strata. Moreover, the metric should be euclidean near the critical points and the gradient vector field should be {\em tangential} to the strata of the boundary. In addition, the flow should flow from higher dimensional strata to lower dimensional strata, but never from lower dimensional strata to higher dimensional ones. 

%%%%%%%%%%%%%%%%%%%%%%%%%%%%%%%%%%%%%%%%%%%%%%%%%%%%%%%%%%%%%%%%%%%%%%%%%%%%%%%%%%%%%%%%%%%%%
%%%%%%% subsubsection 

\subsubsection*{{\bf 0-strata for $f_1$:}}

First recall that $\Crit(f_0)$ is a partially ordered set: We have $x_0 \geq y_0$ if and only if there is a flow line from $x_0$ to $y_0$. Here we also allow the stationary flow line in order to have $x_0 \geq x_0$. If we use the notation $x_0 > y_0$ we assume $x_0 \neq y_0$. If $\Ind(x_0)-\Ind(z_0)=1$, the space $\mcMhat(x_0, z_0, f_0) $ is zero dimensional (and compact) and thus a finite union of points. 
Now choose values for $f_1$ on the zero dimensional moduli spaces such that for mutually distinct $x_0 \geq y_0 \geq z_0$ we have 
\beqs
f_1|_{\mcMhat(x_0, y_0, f_0)}=:f_{1 \left[ \begin{smallmatrix} x_0 \\ y_0 \end{smallmatrix} \right]}
> 
f_{1 \left[ \begin{smallmatrix} y_0 \\ z_0 \end{smallmatrix} \right]} := f_1|_{\mcMhat(y_0, z_0, f_0)} >0
\eeqs
where we choose the same value for all connected components of a moduli space. Note that we require the function to be strictly positive. The reason will become apparent later. Moreover assume that different moduli spaces have different values for $f_1$. Since there are only finitely many critical points we can achieve this easily. 

Let us remark that the zero dimensional boundary strata are corners. By construction, they will turn out to be critical points. Note that Akaho \cite{akaho} also constructs his Morse function to be nonzero at the critical points on the boundary.

To the moduli space of stationary curves $\mcMhat(x_0, x_0, f_0)$, we formally assign the value $0$ to $f_1$. (Note that some people consider $\mcMhat(x_0, x_0, f_0) $ as the empty set.) 

The $0$-strata are automatically critical points since the gradient has to be tangential to all strata meeting at a $0$-dimensional stratum. This is only possible if the gradient vanishes. In accordance with this, we define the metric
\beqs
g_{1\left[\begin{smallmatrix} x_0 \\ y_0 \end{smallmatrix} \right]}:= g_1|_{\mcMhat(x_0, y_0, f_0)}:=g_{eucl}
\eeqs
to be euclidean.

%%%%%%%%%%%%%%%%%%%%%%%%%%%%%%%%%%%%%%%%%%%%%%%%%%%%%%%%%%%%%%%%%%%%%%%%%%%%%%%%%%%%%%%%%%%%%%%%%%%%%%
%%%%%%  subsubsection

\subsubsection*{{\bf Morse functions on cartesian products:}}

The product structure of the boundary of a compactified Morse moduli space suggests to choose a Morse function on the moduli space which is naturally compatible with the product structure of the boundary. Such Morse functions are obtained as follows:
Let $A$ and $B$ be smooths $n_A$- resp. $n_B$-dimensional manifolds equipped with Morse functions $f_A$ and $f_B$ and associated metrics $g_A$ and $g_B$. On the product $C:= A\x B$, consider $f\colon C \to \R$, $f(a, b):=f_A(a) + f_B(b) $ and $g:=g_A \oplus g_B$. It holds
\beqs
\grad_g f=(\grad_{g_A}f_A, \grad_{g_B}f_B)
\eeqs
and therefore $(a,b) \in \Crit(f)$ if and only if $a \in \Crit(f_A)$ and $b \in \Crit(f_B)$. We have $\Ind(a, b)=\Ind(a)+\Ind(b)$ and thus the set $\Crit_k(f)$ of critical points of $f$ with index $k$ is given by
\beqs
\Crit_k(f)= \{(a,b) \in \Crit(f_A) \x \Crit(f_B) \mid \Ind(a)+ \Ind(b)=k\}.
\eeqs
The equation $\gadot(t)=\grad_g f(\ga(t))$ splits into 
\beqs
\left\{
\begin{aligned}
& \gadot_A(t)=\grad_{g_A}f_A(\ga_A(t)), \\
& \gadot_B(t)=\grad_{g_B}f_B(\ga_B(t))
\end{aligned}
\right.
\eeqs
and thus there is a flow line $(\ga_A, \ga_B)$ from $(a_-,b_-) \in \Crit(f)$ to $(a_+, b_+) \in \Crit(f)$ if and only if there are flow lines $\ga_A$ and $\ga_B$ from $a_-$ to $a_+$ and from $b_-$ to $b_+$. In particular, this requires $\Ind(a_-) \geq \Ind(a_+)$ and $\Ind(b_-) \geq \Ind(b_+)$. In order to obtain flow lines of relative index 1, either $\ga_A$ or $\ga_B$ has to be constant.

Note that one cannot simply identify $\mcMhat((a_-, b_-), (a_+, b_+),f)$ with $\mcMhat(a_-, a_+, f_A) \x \mcMhat(b_-, b_+, f_B)$. The first space has dimension 
\beqs
(\Ind(a_-) + \Ind(b_-))-(\Ind(a_+)- \Ind(b_+) )-1
\eeqs
and the second one $(\Ind(a_-) + \Ind(b_-))-(\Ind(a_+)- \Ind(b_+) )-2$. 
Given trajectories $\ga_A$ and $\ga_B$, also $\ga_A^\si(t):=\ga_A(t+\si)$ and $\ga_B^\tau(t):=\ga_B(t+\tau)$ are Morse trajectories. Thus $\ga^{\si, \tau}(t):=(\ga_A^\si (t), \ga^\tau_B(t))$ is a solution, but changing $(\si, \tau)$ changes the geometric shape of $\ga$ and not its parametrisation.

%%%%%%%%%%%%%%%%%%%%%%%%%%%%%%%%%%%%%%%%%%%%%%%%%%%%%%%%%%%%%%%%%%%%%%%%%%%%%%%%%%%%%%%%
%%%%%%%%%%%  new subsubsection

\subsubsection*{{\bf Morse index on strata:}}

Let $M$ be a $\langle k \rangle$-manifold with faces $\del_1 M$, \dots, $\del _k M$. Let $\ep =(\ep_1, \dots, \ep_k) \in \{0, 1\}^k$ and set $M(\ep):= \bigcap _{\ep_i=0}\del_i M$ with $M(1, \dots, 1):=M$. Let $f\colon M \to \R$ be a Morse function. For $\ep$, $\de \in \{0, 1\}^k$ we write $\ep \geq \de$ if $\ep_i \geq \de_i$ for all $1 \leq i \leq k$.

Let $\ep$, $\de \in \{0, 1\}^k$ with $\de \geq \ep$ and $x \in M(\ep) \subset M$ a critical point. Then we define the Morse index $\Ind_{\de \geq \ep}(x)$ of $x \in M(\ep) \subset M(\de)$ as the number of negative eigenvalues of $Df(x)$ in $T_xM(\de)$ and abbreviate $\Ind_\ep(x):=\Ind_{\ep \geq \ep}(x)$.

If there are no critical points in a small enough collar neighbourhood of each strata and if the Morse function is constructed radially, one deduces at once:

\begin{remark}
\label{relativeMorseIndex}
If the negative gradient flow of a Morse function flows from higher to lower dimensional strata then $\Ind_\ep(x)=\Ind(x)$ for all $\ep \in \{0,1 \}^k$ and $x \in \Crit(M)$.
\end{remark}

%%%%%%%%%%%%%%%%%%%%%%%%%%%%%%%%%%%%%%%%%%%%%%%%%%%%%%%%%%%%%%%%%%%%%%%%%%%%%%%%%%%%%%%%%%
%%%%%%%%%%  new subsubsection

\subsubsection*{{\bf l-strata for $f_1$ with $l \geq 1$}}

For higher dimensional strata, we define $f_1 $ recursively:
Consider $x_0$, $z_0 \in \Crit(f_0)$ with $\Ind(x_0) - \Ind(z_0) = l+1$ which implies $\dim \mcMhat(x_0, z_0, f_0)=l$ for $l \geq 1$ in case $\mcMhat(x_0, z_0, f_0 ) \neq \emptyset$. Assume that we already defined the Morse function $f_1$ and $f_1$-euclidean metric $g_1$ in the suitable manner on all $0$-, \dots, $(l-1)$-strata.
There are two cases:

\vspace{2mm}

{\em Case 1:} The boundary is empty: If $\del\mcMhat(x_0, z_0, f_0)=\emptyset$, then there are no restrictions on the choice of $f_{1 \left[ \begin{smallmatrix}x_0 \\ z_0 \end{smallmatrix} \right]}:= f_1|_{\mcMhat(x_0, z_0, f_0)}$ apart from it being larger than on the lower dimensional strata. And the only restriction on $g_{1 \left[ \begin{smallmatrix}x_0 \\ z_0 \end{smallmatrix} \right]}:= g_1|_{\mcMhat(x_0, z_0, f_0)}$ is that it has to be euclidean near the critical points of $f_{1 \left[ \begin{smallmatrix}x_0 \\ z_0 \end{smallmatrix} \right]}$.

\vspace{2mm}

{\em Case 2:} The boundary is not empty: If 
\beqs
\emptyset \neq \del\mcMhat(x_0, z_0, f_0)= \bigcup_{y_0 \in \Crit(f_0) } \mcMhat(x_0, y_0, f_0) \x \mcMhat(y_0, z_0, f_0) ,
\eeqs
then the highest dimensional stratum in $\del\mcMhat(x_0, z_0, f_0)$ is an $(l-1)$-stratum. In $\mcMhat(x_0, y_0, f_0)$, the highest dimensional one is an $(l-1-k)$-stratum and, on the space $\mcMhat(y_0, z_0, f_0)$, the highest dimensional stratum is an $k$-stratum where $0 \leq k \leq l-1$. We already have Morse functions $f_{1\left[ \begin{smallmatrix}x_0 \\ y_0 \end{smallmatrix} \right]}$ and $f_{1\left[ \begin{smallmatrix}y_0 \\ z_0 \end{smallmatrix} \right]}$ and metrics $g_{1\left[ \begin{smallmatrix}x_0 \\ y_0 \end{smallmatrix} \right]}$ and $g_{1\left[ \begin{smallmatrix}y_0 \\ z_0 \end{smallmatrix} \right]}$ on 
$\mcMhat(x_0, y_0, f_0)$ and $\mcMhat(y_0, z_0, f_0) $. For coordinates $(a,b) \in \mcMhat(x_0, y_0, f_0) \x \mcMhat(y_0, z_0, f_0) \subset \del\mcMhat(x_0, z_0, f_0)$ we define 
\beqs
f_1|_{\del \mcMhat(x_0, z_0, f_0)}(a,b)
:= f_{1\left[ \begin{smallmatrix}x_0 \\ y_0 \end{smallmatrix} \right]}(a)
+
f_{1\left[ \begin{smallmatrix}y_0 \\ z_0 \end{smallmatrix} \right]}(b)
\eeqs
on the connected components. Since we required the Morse function to be positive on the zero dimensional moduli spaces, by induction, we have $f_1|_{\del \mcMhat(x_0, z_0, f_0)}(a,b)
>  f_{1\left[ \begin{smallmatrix}x_0 \\ y_0 \end{smallmatrix} \right]}(a)$ and $f_1|_{\del \mcMhat(x_0, z_0, f_0)}(a,b) > f_{1\left[ \begin{smallmatrix}y_0 \\ z_0 \end{smallmatrix} \right]}(b)$.
The purpose is that the value of $f_1$ rises when we pass from lower to higher dimensional strata since we want the negative gradient flow to flow strictly from higher strata to lower ones.

Analogously we define the metric via 
\beqs
g_1|_{\del \mcMhat(x_0, z_0, f_0)}
:= g_{1\left[ \begin{smallmatrix}x_0 \\ y_0 \end{smallmatrix} \right]}
\oplus
g_{1\left[ \begin{smallmatrix}y_0 \\ z_0 \end{smallmatrix} \right]}.
\eeqs
Note that the product of two euclidean metrics is euclidean. Thus $g_1|_{\del \mcMhat(x_0, z_0, f_0)}$ is euclidean near the critical points of $f_1|_{\del \mcMhat(x_0, z_0, f_0)}$.

\vspace{2mm}

Now we extend the Morse function and the metric to the interior of $\mcMhat(x_0, y_0, f_0)$ such that the gradient vector field is tangential to the boundary, i.e. tangential to every stratum, and the metric is euclidean near the critical points. Moreover, the extension of the Morse function is forbidden to have critical points in an small collar neighbourhood of the boundary (on the boundary, it may have critical points).
Explicit examples of this type of construction can be found in Ludwig \cite{ludwig} and Akaho \cite{akaho}. One chooses a collar neighbourhood of the boundary and considers its normal bundle. In the interior of a face, one can follow Akaho's \cite{akaho} construction for manifolds with boundary. He gives an explicit formula for the Morse function. Recall that a corner (or lower boundary stratum) of a $\langle k \rangle$-manifold is modeled on a quadrant (or half space) of $\R^n$ such that one can construct a suitable Morse function `by hand'. Then patch it together.
Denote the new Morse function by $f_{1\left[ \begin{smallmatrix}x_0 \\ z_0 \end{smallmatrix} \right]}:= f_1|_{\mcMhat(x_0, z_0, f_0)}$.
Proceed analogously with the metric and make it euclidean near the critical points of the Morse function and obtain $g_{1\left[ \begin{smallmatrix}x_0 \\ z_0 \end{smallmatrix} \right]}:= g_1|_{\mcMhat(x_0, z_0, f_0)}$.

\vspace{2mm}

This finishes the construction of $f_1$ on Morse moduli spaces of $f_0$: The Morse function $f_1$ takes the cartesian product structure of the boundary into account, behaves naturally with restriction to lower dimensional strata and its negative gradient flow flows strictly from higher strata to lower strata.

%%%%%%%%%%%%%%%%%%%%%%%%%%%%%%%%%%%%%%%%%%%%%%%%%%%%%%%%%%%%%%%%%%%%%%%%%%%%%%%%%%%%%%%%%%%%%%%
%%%%%%% new subsubsection 

\subsubsection*{{\bf The construction of $f_2$ on the moduli spaces of $f_1$:}}

Assume for a moment that the negative gradient flow of the Morse function does not only flow from higher to lower, but {\em also} from lower to higher strata as sketched below in Figure \ref{breaking} (a).

\begin{figure}[h]

\begin{center}

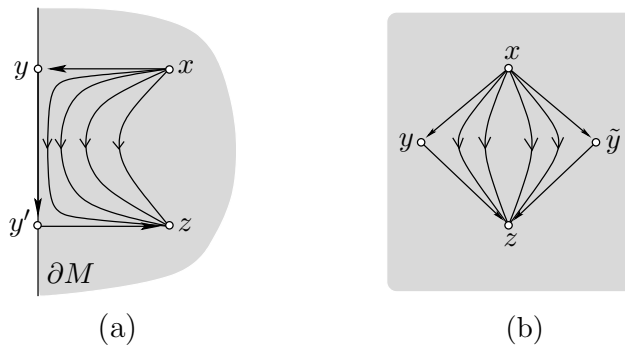
\caption{Breaking of trajectories: (a) on the boundary, (b) in the interior.}
\label{breaking}
\end{center}

\end{figure}

As sketched in Figure \ref{breaking} (a), a trajectory between two critical points $x$ and $z$ in the interior of the manifold with index difference $\Ind(x) - \Ind(z)=1$ may actually break via the boundary into {\em three} trajectories instead of only {\em two} without involvement of the boundary (cf. Figure \ref{breaking} (b)). This phenomenon is explained in Akaho \cite{akaho} and Kronheimer $\&$ Mrowka \cite{kronheimer-mrowka}. It is due to the following observation. W.l.o.g let $y' \in \del_1M=:M(\ep)$ with $\ep=(0,1, \dots, 1)$ and $M=M(\de)$ with $\de=(1, \dots, 1)$. Then $\Ind_\ep(y') < \Ind_{\ep \leq \de}(y')$, i.e. it matters if we consider $y'$ as critical point of the Morse function as function restricted to the boundary or as function on $M$.

\vspace{2mm}

The boundary of the moduli spaces $\mcMhat(x,z)$ in Figure \ref{breaking} (a) has still product structure as shown in Akaho \cite{akaho}. But for the $n$-category structure in the later sections we need to be able to `compose' (i.e. glue) two `connecting' trajectories --- and we cannot glue just two of the three parts of the broken trajectories in Figure \ref{breaking} (a) since there {\em is no} trajectory from $x$ to $y'$ and also none from $y$ to $z$. Therefore we need to exclude such situations if we want to define an $n$-category later on.

\vspace{2mm}

This dilemma is solved by using \refrelativeMorseIndex: Impose the assumption that the Morse function is increasing from lower dimensional strata to higher dimensional ones. Equivalently, the negative gradient flow flows strictly from higher to lower strata and never from lower to higher strata. Then the Morse index does not depend on the strata and we do not have multiple breaking of trajectries of index difference one. An analogous statement holds true for critical points with higher index difference. We conclude

\begin{theorem}
\label{morseOnManifoldsWithCorners}
Let $f$ be a Morse function on a compact $\langle k \rangle$-manifold whose negative gradient flow flows from higher to lower strata, but not from lower to higher ones. Assume the metric to be euclidean near the critical points. 
 Let $x$, $z \in \Crit(f)$ with $x>z$. Then there exists $k \in \N_0$ such that $\mcMhat(x,z)$ is an $(\Ind(x)-\Ind(z)-1)$-dimensional $\langle k \rangle$-manifold with corners and its boundary is given by 
\beqs
\del \mcMhat(x,z) = \bigcup_{\stackrel{ (\Ind(x) - \Ind(z) -1) \geq l \geq 0}{x > y_1 > \dots > y_l >z}} \mcMhat(x,y_1) \x\mcMhat(y_1, y_2) \x \dots \x \mcMhat(y_{l-1}, y_l) \x \mcMhat(y_l, z)
\eeqs
where $y_1$, \dots, $y_l \in \Crit(f)$. There is a canonical smooth structure on $\mcMhat(x,z)$.
\end{theorem}

\begin{proof}
For $k=1$ and the situation of Figure \ref{breaking} (a) with $\Ind(x)-\Ind(z)\leq 2$ this has been proven by Akaho \cite{akaho}. On manifolds with corners and $\Ind(x)-\Ind(z)\leq 2$ the result is implied by Ludwig \cite{ludwig}. The general case goes analogously since \refrelativeMorseIndex\ reduces the possible breaking phenomena to the case of a closed smooth manifold where \refmodulispacecompact\ applies.
\end{proof}

Altogether, we have so far constructed a Morse function $f_1$ with certain properties on the Morse moduli spaces of $f_0$. Now we consider the critical points of $f_1$ and the Morse moduli spaces between them. \refmorseOnmanifoldsWithCorners\ states that we are again dealing with $\langle k \rangle$-manifolds for certain $k \in \N_0$. So we can repeat the construction of a radial Morse function compatible with lower strata now on the Morse moduli spaces of $f_1$, i.e. there exists a Morse function $f_2$ with similar properties as $f_1$ on the Morse moduli spaces $\mcMhat(x_1, z_1, f_{1 \left[ \begin{smallmatrix}x_0 \\ z_0 \end{smallmatrix} \right]} )$ of $f_1$ where $x_1$, $z_1 \in \Crit{f_{1 \left[ \begin{smallmatrix}x_0 \\ z_0 \end{smallmatrix} \right]}}$ and $x_0$, $z_0 \in \Crit(f_0)$.

\vspace{2mm}

Once $f_2 $ is constructed, its Morse moduli spaces are again $\langle k \rangle$-manifolds according to \refmorseOnmanifoldsWithCorners. By construction of the compactified Morse moduli spaces, we always decrease the dimension by (at least) one. Therefore this iteration process terminates after a finite number of repetitions when the moduli spaces become zero dimensional.

\vspace{2mm}

Summarizing the above paragraphs, we constructed Morse functions on the compactified Morse module spaces of Morse functions in such a way that each Morse function $f_i$ is compatible with its restriction to lower strata of the Morse moduli spaces of $f_{i-1}$ which have the structure of a cartesian product.

%%%%%%%%%%%%%%%%%%%%%%%%%%%%%%%%%%%%%%%%%%%%%%%%%%%%%%%%%%%%%%%%%%%%%%%%%%%%%%%%%%%%%%%%%%%%%%%%%%%%%%
%%%%%%%%%%%%%%%%% new section %%%%%%%%%%%%%%%%%%%%%%%%%%%%%%%%%%%%%%%%%%%%%%%%%%%%%%%%%%%%%%%%%%%%%%
%%%%%%%%%%%%%%%%%%%%%%%%%%%%%%%%%%%%%%%%%%%%%%%%%%%%%%%%%%%%%%%%%%%%%%%%%%%%%%%%%%%%%%%%%%%%%%%%%%%%%%

\section{(Almost) strict $n$-categories}

The definition of strict $n$-categories is due to Charles Ehresmann. We recall their definition from Leinster's book \cite{leinster} which gives two equivalent ways of defining strict $n$-categories. One can define it either recursively via enriched categories or direct by listing six properties which have to be satisfied. We focus on the latter definition.

\vspace{2mm}

The usual definition of a category will blend into this framework as a $1$-category. Since a category can be considered as a directed graph with structure the way of defining $n$-categories starts as follows.

\begin{definition}
\label{nglobular}
Given $n \in \N$, we define an {\em $n$-globular set} $Y$ to be a collection of sets $\{Y(l) \mid {0 \leq l \leq n} \}$ together with {\em source} and {\em target functions} $s$, $t: Y(l) \to Y(l-1)$ for $1 \leq l \leq n$ satisfying $s \circ s = s \circ t$ and $t \circ s = t \circ t$. Elements $A_l \in Y(l)$ are called {\em $l$-cells}.
\end{definition}

To visualize $n$-globular sets, one can think of the $l$-cells as $l$-dimensional disks and sketch them accordingly:
a $0$-cell $A_0 \in Y(0)$ is displayed as a point
\beqs
\bullet \ A_0
\eeqs
$0$-cells are sometimes also called {\em objects}.
A $1$-cell $A_1 \in Y(1)$ with $s(A_1)=A_0 \in Y(0)$ and $t(A_1)=B_0 \in Y(0)$ is sketched as an arrow (or $1$-disk) connecting the $0$-cells $A_0$ and $B_0$

\begin{figure}[H]

\begin{center}

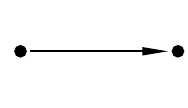
\end{center}

\end{figure}

Sometimes $1$-cells are also called {\em morphisms}.
A $2$-cell (sometimes called {\em morphism between morphisms}) $A_2 \in Y(2)$ with $s(A_2)=A_1$, $t(A_2)=B_1 \in Y(1)$ and therefore $s(A_1)=s(B_1)=:A_0$ and $t(A_1)=t(B_1)=:B_0$ is sketched as an double arrow or $2$-disk connecting the $1$-cells $A_1$ and $B_1$

\begin{figure}[H]

\begin{center}

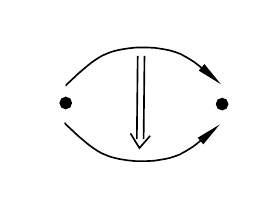
\end{center}

\end{figure}

A $3$-cell $A_3 \in Y(3)$ with $s(A_3)=A_2$ and $t(A_3)=B_2$ is sketched as a $3$-disk or triple arrow `perpendicular to the sheet of paper'

\begin{figure}[H]

\begin{center}

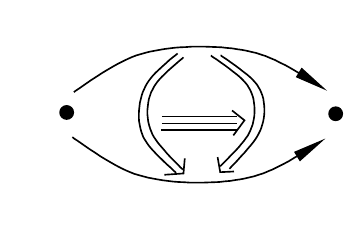
\end{center}

\end{figure}

Generally, $l$-cells are represented by an $l$-arrow or a $l$-disk, although sketches clearly reach their limits.

\vspace{2mm}

Given $1$-cells $A_1$, $B_1 \in Y(1)$ with `matching' source and target conditions $s(A_1)=:A_0$, $t(A_1)=s(B_1)=: B_0$ and $t(B_1)=:C_0$ we are clearly tempted to `compose' $A_1$ and $B_1$ `along' $B_0$ like usual morphisms.

\begin{figure}[H]

\begin{center}

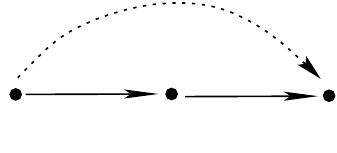
\end{center}

\end{figure}

Now consider two $2$-cells $A_2$, $B_2 \in Y(2)$: There are two different `matching' conditions possible: On the one hand, we might have $t(A_2)=s(B_2)=:B_1$, i.e. we would like to compose $A_2$ and $B_2$ `along' the $1$-cell $B_1$ as sketched in

\begin{figure}[H] % ! verstaerkt Vorrang des Parameters

\begin{center}

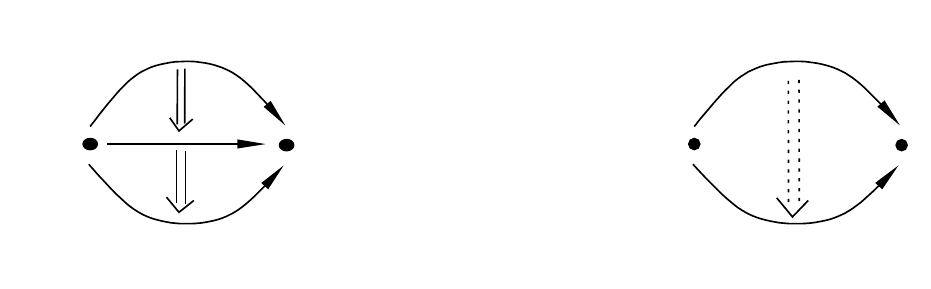
\end{center}

\end{figure}

But, on the other hand, we also can have a matching condition along a $0$-cell as sketched in

\begin{figure}[H] % H ist 'absolut hier', braucht \usepackage{float}

\begin{center}

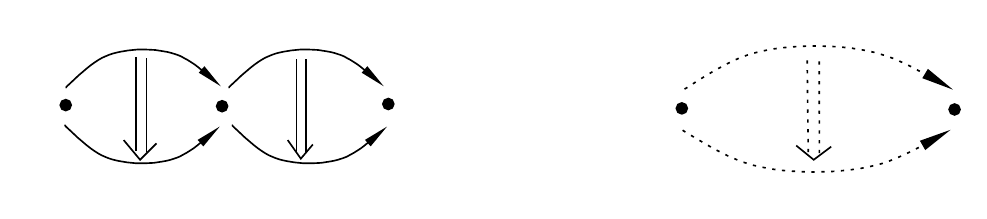
\end{center}

\end{figure}

which also suggests a composition if we `first' compose the $1$-cells $A_1$ with $C_1$ and $B_1$ with $C_1$.
In fact, as we will see later, it will turn out that there are $l$ possible ways to compose $l$-cells, namely along $0$-cells, $1$-cells, \dots, $(l-1)$-cells. Given an $n$-globular set $Y$, we express the matching conditions by means of the set
\beqs
Y(l) \x_p Y(l) := \{(\mfxti, \mfx) \in Y(l) \x Y(l) \mid s^{l-p}(\mfxti)=t^{l-p}(\mfx)\}
\eeqs
$0 \leq p < l \leq n$. More precisely, $Y(l) \x_p Y(l)$ is the set of $l$-cells which can be composed along a $p$-cell.

\vspace{2mm}

Another important feature are {\em identity functions} on the $n$-globular set, i.e. a collection of functions $ \mathbf{1}: Y(l) \to Y(l+1) $ for $0 \leq l \leq n-1$ which assign to a $l$-cell $A_l\in Y(l)$ a certain $(l+1)$-cell ${\bf 1}_{A_l}$ with source and target $A_l$. For $0$-cells, this means

\begin{figure}[H]

\begin{center}

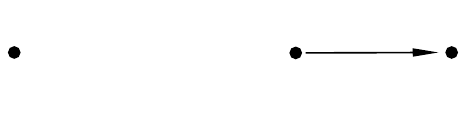
\end{center}

\end{figure} 	

And for $1$-cells, this leads to

\begin{figure}[H]

\begin{center}

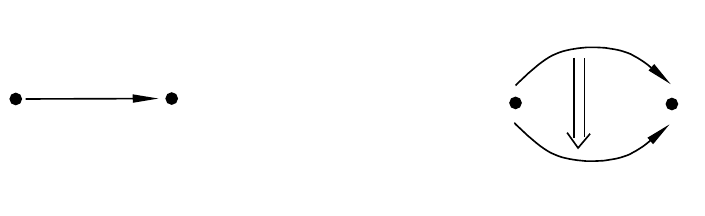
\end{center}

\end{figure}

In the following definition, we will pose additional conditions on the composite and identities.

\begin{definition}
\label{ncategory}
Let $n \in \N$. A {\bf strict $n$-category $\mcY$} is an $n$-globular set $Y$ equipped with
\begin{itemize}
 \item 
a function $\circ_p\colon Y(l) \x_p Y(l) \to Y(l)$ for all $0 \leq p < l\leq n$. We set $\circ_p(C_l, A_l) =:C_l \circ_p A_l$ and call it {\bf composite} of $A_l$ and $C_l$.
\item
a function $\mathbf{1}\colon Y(l) \to Y(l+1) $ for all $0 \leq l < n$. We set $\mathbf{1}_{A_l}:=\mathbf{1}(A_l)$ and call it the {\bf identity} on $A_l$.
\end{itemize}
These have to satisfy the following axioms:
\begin{enumerate}[(a)]
 \item 
{\bf (Sources and targets of composites)} For $0 \leq p < l \leq n$ and $(C_l, A_l) \in Y(l) \x_p Y(l) $ we require
\begin{align*}
& \mbox{for } p=l-1 && s(C_l \circ_p A_l)=s(A_l) && \mbox{and} && t(C_l \circ_p A_l) = t(C_l), \\ 
& \mbox{for } p \leq l-2 && s(C_l \circ_p A_l)= s(C_l) \circ_p s(A_l) && \mbox{and} && t(C_l \circ_p A_l)=t(C_l) \circ_p t(A_l).
\end{align*}
\item
{\bf (Sources and targets of identities)} For $0 \leq l <  n$ and $A_l \in Y(l)$ we require 
\beqs
s(\mathbf{1}_{A_l})=A_l=t(\mathbf{1}_{A_l}).
\eeqs
\item
{\bf (Associativity)}  
For $0 \leq p < l \leq n$ and $A_l$, $C_l$, $E_l\in Y(l)$ with $(E_l, C_l)$, $(C_l, A_l) \in Y(l) \x_p Y(l)$ we require
\beqs
(E_l\circ_p C_l) \circ_p A_l = E_l \circ_p( C_l \circ_p A_l).
\eeqs 
\item {\bf (Identities)}
For $0 \leq p < l \leq n$ and $A_l \in Y(l)$ we require 
\beqs
\mathbf{1}^{l-p}(t^{l-p}(A_l)) \circ_p A_l = A_l = A_l \circ_p \mathbf{1}^{l-p}(s^{l-p}(A_l)).
\eeqs
\item {\bf (Binary interchange)}
For $0 \leq q < p < l \leq n$ and $A_l$, $C_l$, $E_l$, $H_l \in Y(l)$ with
\beqs
(H_l, E_l), (C_l, A_l) \in Y(l) \x_p Y(l) \mbox{ and } (H_l, C_l), (E_l,A_l) \in Y(l)\x_q Y(l)
\eeqs
we require
\beqs
(H_l \circ_p E_l) \circ_q (C_l \circ_p A_l) = (H_l \circ_q C_l) \circ_p (E_l \circ_q A_l).
\eeqs
\item {\bf (Nullary interchange)}
For $0 \leq p < l< n$ and $(C_l, A_l) \in Y(l) \x_p Y(l)$ we require $\mathbf{1}_{C_l} \circ_p \mathbf{1}_{A_l} = \mathbf{1}_{C_l \circ_p A_l}$.
\end{enumerate}
If $\mcY$ and $\mcZ$ are strict $n$-categories we define a {\bf strict $n$-functor $f$} as a map $f \colon Y \to Z$ of the underlying $n$-globular sets commuting with composition and identities. This defines a category {\bf Str-$n$-Cat} of strict $n$-categories.
\end{definition}

Slightly relaxing the requirements, we define

\begin{definition}
An {\bf almost strict $n$-category} satisfies the requirements of a strict $n$-category up to canonical isomorphism.
\end{definition}

The compatibility of the identities with the source and target functions in item (d) of \refncategory\ can be visualized via

\begin{figure}[H] %  verstaerkt Vorrang des Parameters, aber nicht so stark wie [H]

\begin{center}

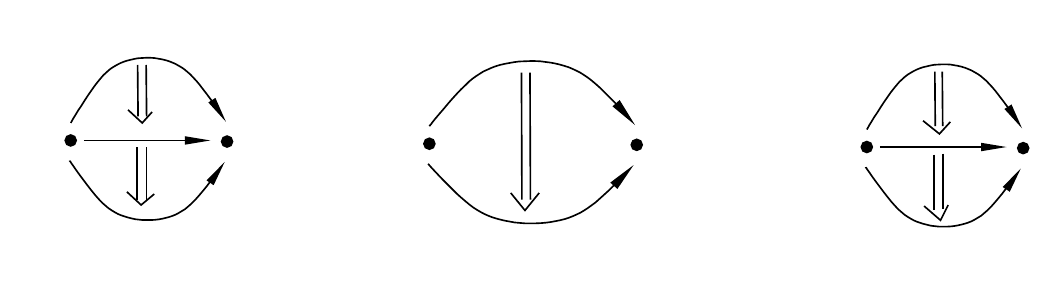

\end{center}

\end{figure}

The binary interchange of item (e) in \refncategory\ can be sketched as

\begin{figure}[H] %  verstaerkt Vorrang des Parameters, aber nicht so stark wie [H]

\begin{center}

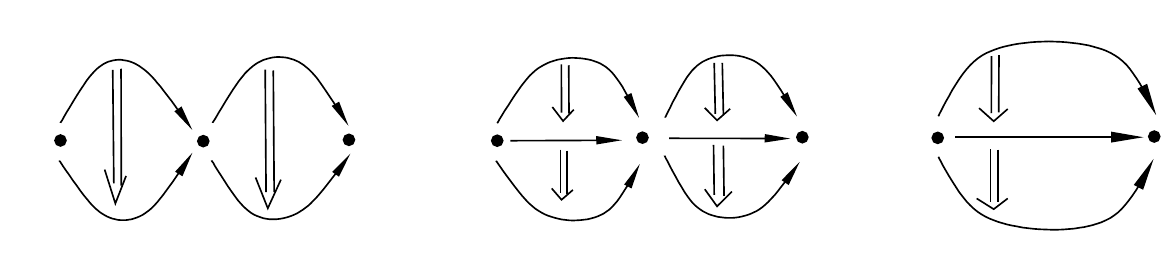

\end{center}

\end{figure}

And the nullary interchange looks like

\begin{figure}[H] 

\begin{center}

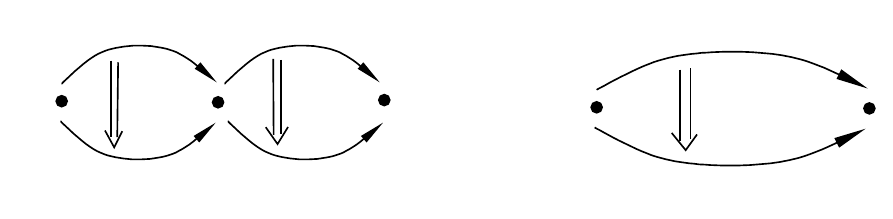

\end{center}

\end{figure}

%%%%%%%%%%%%%%%%%%%%%%%%%%%%%%%%%%%%%%%%%%%%%%%%%%%%%%%%%%%%%%%%%%%%%%%%%%%%%%%%%%%
%%%%%%%%%%  new section  %%%%%%%%%%%%%%%%%%%%%%%%%%%%%%%%%%%%%%%%%%%%%%%%%%%%%%%%%
%%%%%%%%%%%%%%%%%%%%%%%%%%%%%%%%%%%%%%%%%%%%%%%%%%%%%%%%%%%%%%%%%%%%%%%%%%%%%%%%%%%

\section{The $n$-category of Morse trajectory spaces}

\label{sectionmorsencat}

%%%%%%%%%%%%%%%%%%%%%%%%%%%%%%%%%%%%%%%%%%%%%%%%%%%%%%%%%%%%%%%%%%%%%%%%%%%
%%%%%%%  subsection

\subsection{$n$-globular set of Morse moduli spaces}

In the following, we will define the $n$-globular set $X$ of Morse moduli spaces on which the $n$-category of Morse moduli spaces $\mcX$ is based.

\vspace{2mm}

Let $M$ be a compact $n$-dimensional $\langle k \rangle$-manifold $M$ with a Morse function $f_0$ (constructed as in the previous section) and a $f_0$-euclidean metric $g_0$. We set 
\beqs
X(0):=\{x_0 \mid x_0 \in \Crit(f_0)\}.
\eeqs
Given two critical points $x_0$, $y_0 \in \Crit(f_0)$, we consider the space $\mcMhat(x_0, y_0, f_0)$. On this space, we choose a Morse function
$f_{1 \left[ \begin{smallmatrix} x_0 \\ y_0 \end{smallmatrix}\right]}$ with $f_{1 \left[ \begin{smallmatrix} x_0 \\ y_0 \end{smallmatrix}\right]}$-euclidean metric $g_{1 \left[ \begin{smallmatrix} x_0 \\ y_0 \end{smallmatrix}\right]}$ as described in \refmodulicorners. We define
\beqs
X(1):=\{(x_1, \mcMhat(x_0, y_0, f_0)) \mid x_0, y_0 \in \Crit(f_0), \ x_1 \in \Crit(f_{1 \left[ \begin{smallmatrix} x_0 \\ y_0 \end{smallmatrix}\right]})\}.
\eeqs
The index of the Morse function $f_{1 \left[ \begin{smallmatrix} x_0 \\ y_0 \end{smallmatrix}\right]}$ or metric $g_{1 \left[ \begin{smallmatrix} x_0 \\ y_0 \end{smallmatrix}\right]}$ starts with the number of the level on which the function or metric lives and continues with the (history of) critical points which gave rise to the moduli space. The upper row states the source points and the lower row the target points. It is important to keep carefully track of the `history' of a moduli space.
Analogously, given $x_1$, $y_1 \in \Crit(f_{1 \left[ \begin{smallmatrix} x_0 \\ y_0 \end{smallmatrix}\right]})$, choose a Morse function 
$f_{2 \left[ \begin{smallmatrix} x_1, x_0 \\ y_1, y_0 \end{smallmatrix}\right]}$ and 
$f_{2 \left[ \begin{smallmatrix} x_1, x_0 \\ y_1, y_0 \end{smallmatrix}\right]}$-euclidean metric $g_{2 \left[ \begin{smallmatrix} x_1, x_0 \\ y_1, y_0 \end{smallmatrix}\right]}$ on 
$\mcMhat(x_1, y_1, f_{1 \left[ \begin{smallmatrix} x_0 \\ y_0 \end{smallmatrix}\right]})$ and let 
\beqs
X(2):= \left\{
\left(x_2, \mcMhat(x_1, y_1, f_{1 \left[ \begin{smallmatrix} x_0 \\ y_0 \end{smallmatrix}\right]})\right)
\left|
\begin{aligned}
& x_0, y_0 \in \Crit(f_0), \\
& x_1, y_1 \in \Crit(f_{1 \left[ \begin{smallmatrix} x_0 \\ y_0 \end{smallmatrix}\right]}), \\
& x_2 \in \Crit(f_{2 \left[ \begin{smallmatrix} x_1, x_0 \\ y_1, y_0 \end{smallmatrix}\right]})
\end{aligned}
\right.
\right\}.
\eeqs
We work with tupels (point, moduli space) instead of only the moduli spaces in order to obtain well-defined source and target function.
We iterate this process
and obtain for $2 \leq l \leq n$
\beqs
X(l):=
\left\{
\left(x_l, \mcMhat(x_{l-1}, y_{l-1}, 
f_{l-1 
\left[
\begin{smallmatrix}
x_{l-2}, \dots, x_0 \\
y_{l-2}, \dots, y_0
\end{smallmatrix}
\right]
}
)\right)
\left|
\begin{aligned}
& 0 \leq j \leq l-1, \\
& x_j, y_j \in \Crit(
f_{j 
\left[
\begin{smallmatrix}
x_{j-1}, \dots, x_0 \\
y_{j-1}, \dots, y_0
\end{smallmatrix}
\right]
}
), \\
& x_l \in \Crit(f_
{l 
\left[
\begin{smallmatrix}
x_{l-1}, \dots, x_0 \\
y_{l-1}, \dots, y_0
\end{smallmatrix}
\right]
}
)
\end{aligned}
\right.
\right\}.
\eeqs
Since dividing by the action in the construction of the compactified moduli spaces reduces the dimension by one, we can iterate this procedure at most $n$ times before the moduli spaces in question become zero dimensional and the iteration in turn becomes trivial.  

\vspace{2mm}

For $2 \leq l \leq n$, we define {\em source} and {\em target functions} 
\beqs
s\colon X(l) \to X(l-1) \qquad \mbox{and} \qquad t\colon X(l) \to X(l-1)
\eeqs
via
\begin{align*}
s\left(x_l, \mcMhat(x_{l-1}, y_{l-1}, 
f_{l-1 
\left[
\begin{smallmatrix}
x_{l-2}, \dots, x_0 \\
y_{l-2}, \dots, y_0
\end{smallmatrix}
\right]
}
)\right)
&:= 
\left(x_{l-1}, \mcMhat(x_{l-2}, y_{l-2}, 
f_{l-2 
\left[
\begin{smallmatrix}
x_{l-3}, \dots, x_0 \\
y_{l-3}, \dots, y_0
\end{smallmatrix}
\right]
}
)\right), \\
t\left(x_l, \mcMhat(x_{l-1}, y_{l-1}, 
f_{l-1 
\left[
\begin{smallmatrix}
x_{l-2}, \dots, x_0 \\
y_{l-2}, \dots, y_0
\end{smallmatrix}
\right]
}
)\right)
&:= 
\left(y_{l-1}, \mcMhat(x_{l-2}, y_{l-2}, 
f_{l-2 
\left[
\begin{smallmatrix}
x_{l-3}, \dots, x_0 \\
y_{l-3}, \dots, y_0
\end{smallmatrix}
\right]
}
)\right) \\
\end{align*}
and set for $s, t\colon X(1) \to X(0)$ 
\beqs
s\left (a_1, \mcMhat(x_0, y_0, f_0)\right ):= x_0 \quad \mbox{and} \quad t\left (a_1, \mcMhat(x_0, y_0, f_0)\right):= y_0.
\eeqs
\begin{lemma}
$X:=\{X(l) \mid 0 \leq l \leq n\}$ is an $n$-globular set.
\end{lemma}

\begin{proof}
A short calculation yields $s \circ s= s \circ t$ and $t \circ t=t \circ s$.
\end{proof}

Now we define the $l$-cells which can be composed along $p$-cells:
\beqs
X(l)\x_p X(l):=\{(C_l, A_l) \in X(l) \x X(l) \mid s^{l-p}( C_l)=t^{l-p}(A_l)  \}.
\eeqs
How do this elements look like?
For $l=1$ and $p=0$, an element $(C_1, A_1) \in X(1) \x_0 X(1)$ given by
\beqs
\left((c_1, \mcMhat(c_0, d_0, f_0)), (a_1, \mcMhat(a_0, b_0, f_0))\right) \in X(1) \x_0 X(1)
\eeqs
satisfies $c_0=b_0$. 
More generally, an element $(C_l, A_l) \in X(l) \x_p X(l)$ given by
\beqs
\left( \bigl(c_l,\mcMhat(c_{l-1}, d_{l-1}, f_{l-1 
\left[
\begin{smallmatrix}
c_{l-2}, \dots, c_0 \\
d_{l-2}, \dots, d_0
\end{smallmatrix}
\right]
})\bigr), 
\bigl(a_l,\mcMhat(a_{l-1}, b_{l-1}, f_{l-1 
\left[
\begin{smallmatrix}
a_{l-2}, \dots, a_0 \\
b_{l-2}, \dots, b_0
\end{smallmatrix}
\right]
})\bigr)\right ) 
\eeqs
is characterized by
\begin{equation}
%\label{equalp}
%\newcommand{\refequalp}{(\ref{equalp})}
\left\{
\begin{aligned}
& c_j = a_j && \mbox{ for } 0 \leq j \leq p-1, \\
& d_j = b_j && \mbox{ for } 0 \leq j \leq p-1, \\
& c_p  =b_p. &&
\end{aligned}
\right.
\end{equation}
Therefore we introduce the following more natural notation for tuples $(C_l, A_l) \in X(l) \x_p X(l)$. For $0 \leq j \leq p-1$, we set $a_j=c_j=:\al_j$ and $b_j=d_j=:\be_j$. For the index $p$, we set $a_p=:x_p$, $b_p=c_p =: y_p$ and $d_p=:z_p$. For $p+1 \leq j \leq l$, we keep the $a_j$, $b_j$, $c_j$ and $d_j$. For $(C_l, A_l) \in X(l) \x_p X(l)$, this new notation leads to
\begin{align*}
A_l&= \left(a_l,\mcMhat(a_{l-1}, b_{l-1}, f_{l-1 
\left[
\begin{smallmatrix}
a_{l-2}, \dots, a_{p+1}, x_p, \al_{p-1}, \dots, \al_0 \\
b_{l-2}, \dots, b_{p+1}, y_p, \be_{p-1}, \dots, \be_0
\end{smallmatrix}
\right]
})\right), \\
C_l &= \left( c_l,\mcMhat(c_{l-1}, d_{l-1}, f_{l-1 
\left[
\begin{smallmatrix}
c_{l-2}, \dots, c_{p+1}, y_p, \al_{p-1}, \dots, \al_0 \\
d_{l-2}, \dots, d_{p+1}, z_p, \be_{p-1}, \dots, \be_0
\end{smallmatrix}
\right]
})\right) 
\end{align*}
where one can easily see the meaning of being in $ X(l) \x_p X(l)$: Both $l$-cells arise, up to level $(p-1)$, from the same critical points 
$\left[
\begin{smallmatrix}
\al_{p-1}, \dots, \al_0 \\
\be_{p-1}, \dots, \be_0
\end{smallmatrix}
\right]$.
At level $p$, we have the matching condition 
$\begin{smallmatrix}
\left[
\begin{smallmatrix}
 x_p \\
 y_p
\end{smallmatrix}
\right] \\
\left[
\begin{smallmatrix}
 y_p \\
 z_p
\end{smallmatrix}
\right]
\end{smallmatrix}.
$
There are no additional conditions on the critical points on the higher levels 
$
\begin{smallmatrix}
\left[
\begin{smallmatrix}
a_{l-2} , \dots, a_{p+1} \\
b_{l-2}, \dots, b_{p+1}
\end{smallmatrix}
\right] \\
\left[
\begin{smallmatrix}
c_{l-2}, \dots, c_{p+1} \\
d_{l-2}, \dots, d_{p+1}
\end{smallmatrix}
\right]
\end{smallmatrix}
$
apart from the ones required in the definition of $X(l)$. We call 
$
\left[
\begin{smallmatrix}
a_{l-2}, \dots, a_{p+1}, x_p, \al_{p-1}, \dots, \al_0 \\
b_{l-2}, \dots, b_{p+1}, y_p, \be_{p-1}, \dots, \be_0
\end{smallmatrix}
\right]
$
the {\em history} of $A_l$ up to level $(l-2)$.
In this new notation, it holds for the critical points
\begin{align*}
&\mbox{For } 1 \leq j \leq p-1, &
\al_{j}, \be_{j} & \in \Crit\left(f_{j
\left[
\begin{smallmatrix}
 \al_{j-1},\dots, \al_0 \\
 \be_{j-1}, \dots, \be_0
\end{smallmatrix}
\right]}\right), 
\\
&&
x_p, y_p, z_p & \in \Crit\left(f_{p
\left[
\begin{smallmatrix}
 \al_{p-1},\dots, \al_0 \\
 \be_{p-1}, \dots, \be_0
\end{smallmatrix}
\right]}\right), \\
& \mbox{For } 1 \leq j \leq l-p, &
a_{p+j}, b_{p+j} & \in \Crit\left(f_{p+j
\left[
\begin{smallmatrix}
a_{p+j-1}, \dots a_{p+1}, x_p, \al_{p-1}, \dots, \al_0 \\
b_{p+j-1}, \dots b_{p+1}, y_p, \be_{p-1}, \dots, \be_0
\end{smallmatrix}
\right]}\right),   \\
& \mbox{For } 1 \leq j \leq l-p, &
c_{p+j}, d_{p+j} & \in \Crit\left(f_{p+j
\left[
\begin{smallmatrix}
c_{p+j-1}, \dots, c_{p+1}, x_p, \al_{p-1}, \dots, \al_0 \\
d_{p+j-1}, \dots, d_{p+1}, y_p, \be_{p-1}, \dots, \be_0
\end{smallmatrix}
\right]}\right).  
\end{align*}
If $j=1$ in the two expressions above then there are no $a$'s and $b$'s resp. $c$'s and $d$'s in the index of the function.

%%%%%%%%%%%%%%%%%%%%%%%%%%%%%%%%%%%%%%%%%%%%%%%%%%%%%%%%%%%%%%%%%%%%%%%%%%%%%%%%%%%%%%%
%%%%%%%%%%%%%%  new subsection

\subsection{The identities}

In order to turn the $n$-globular set $X(n)_{n \in \N}$ into an almost strict $n$-category, we need to define the composite and the identities. 

\vspace{2mm}

Let us start with the identities. They are supposed to be functions $\mathbf{1}\colon X(l) \to X(l+1)$ for $0 \leq l \leq n-1$. 
For $l=0$, the set $X(0)$ consists of the critical points $\Crit(f_0)$. Let $x_0 \in X(0)$ 
and identify $x_0$ with the moduli space $\mcMhat(x_0, x_0, f_0)$. Then identify $\mcMhat(x_0, x_0, f_0)$ with the only critical point 
$x_1\in \Crit(f_{1 
\left[ 
\begin{smallmatrix}
x_0 \\                                                                                                                                                                  
x_0                                                                                                                                                                 \end{smallmatrix}
 \right]})$ 
on $\mcMhat(x_0, x_0, f_0) $. Thus we have $x_1 \simeq \mcMhat(x_0, x_0, f_0) \simeq x_0$. With this in mind, we set
\beqs
\mathbf{1}_{x_0}:=\mathbf{1}(x_0):= (x_0, \mcMhat(x_0, x_0, f_0)).
\eeqs
For $l>0$, we set for $A_l= \left(a_l, \mcMhat(a_{l-1}, b_{l-1}, f_{l-1
\left[
\begin{smallmatrix}
 a_{l-2}, \dots, a_0 \\
b_{l-2}, \dots, b_0
\end{smallmatrix}
\right]}
)\right) \in X(l)$
\begin{align*}
\mathbf{1}_{A_l}
& := \mathbf{1}\left (a_l, \mcMhat(a_{l-1}, b_{l-1}, f_{l-1
\left[
\begin{smallmatrix}
 a_{l-2}, \dots, a_0 \\
b_{l-2}, \dots, b_0
\end{smallmatrix}
\right]}
)\right) \\
& := \left(a_l, \mcMhat(a_l, a_l, f_{l
\left[
\begin{smallmatrix}
 a_{l-1}, \dots, a_0 \\
b_{l-1}, \dots, b_0
\end{smallmatrix}
\right]
})\right) \\
& := \left(a_{l+1}, \mcMhat(a_l, a_l, f_{l
\left[
\begin{smallmatrix}
 a_{l-1}, \dots, a_0 \\
b_{l-1}, \dots, b_0
\end{smallmatrix}
\right]
})\right)
\end{align*}
where we again identified $a_{l+1} \simeq a_l$. For $0 \leq l \leq n-1$, this gives us functions 
\beqs
\mathbf{1} \colon X(l) \to X(l+1)  
\eeqs
which will be our candidates for the identity functions of an $n$-category generated by Morse moduli spaces.

%%%%%%%%%%%%%%%%%%%%%%%%%%%%%%%%%%%%%%%%%%%%%%%%%%%%%%%%%%%%%%%%%%%%%%%%%%%%%

\subsection{Motivation for the composite of Morse moduli spaces}

Now we address the composite of the future $n$-category.  
Since this paper also addresses readers from geometry and topology to whom the index consuming and somewhat confusing notation of $n$-categories may be unfamiliar, we will introduce the composite step by step for small $p$ and $l$. Experienced or hurried readers may skip ahead a few pages --- the general formula is given in the next subsection.

%%%%%%%%%%%%%%%%%%%%%%%%%%%%%%%%%%%%%%%%%%%%%%%%%%%%%%%%%%%%%%%%%%%%%%%%%%%%%%%

%\subsubsection*{Case $l=1$ and $p=0$}

\vspace{2mm}

Consider the case $l=1$ and $p=0$. We have 
\begin{align*}
 X(1) \x_0 X(1) 
 =\left\{\left( (c_1, \mcMhat(y_0, z_0, f_0)), (a_1, \mcMhat(x_0, y_0, f_0))\right) 
\left| 
\begin{aligned}
& x_0, y_0, z_0 \in \Crit(f_0),\\
& a_1 \in \Crit(f_{1
\left[ \begin{smallmatrix} x_0 \\ y_0 \end{smallmatrix} \right]}), \\ 
& c_1 \in \Crit(f_{1 \left[ \begin{smallmatrix} y_0 \\ z_0 \end{smallmatrix} \right]})
\end{aligned}
\right.
\right\}
\end{align*}
and define the composite via
\begin{align*}
\left(c_1, \mcMhat(y_0, z_0, f_0)\right) \circ_0 \left(a_1, \mcMhat(x_0, y_0, f_0)\right):= \left((a_1,c_1), \mcMhat(x_0, z_0, f_0)\right).
\end{align*}
In terms of $n$-category language, we composed the two 1-cells $\mcMhat(x_0, y_0,f_0)$, $\mcMhat(y_0, z_0, f_0) \in X(1)$ along the 0-cell $y_0 \in X(0)$:

\begin{figure}[H] %  h verstaerkt Vorrang des Parameters, aber nicht so stark wie [H]

\begin{center}

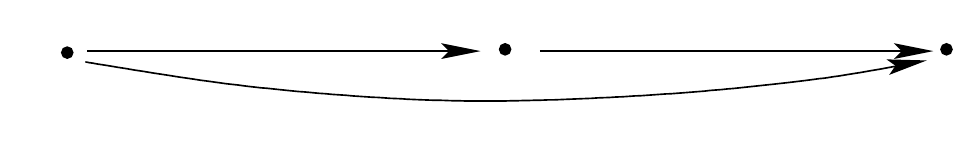

\end{center}

\end{figure}

Geometrically, this describes the gluing procedure of Morse trajectories as lined out for instance in Schwarz \cite{schwarz} and sketched below.

\begin{figure}[H] %  h verstaerkt Vorrang des Parameters, aber nicht so stark wie [H]

\begin{center}

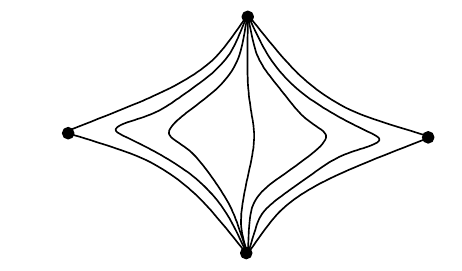

\end{center}

\end{figure}

 $\mcMhat(x_0, y_0,f_0) \x \mcMhat(y_0, z_0, f_0)$ is contained in the boundary of $\mcMhat(x_0, z_0, f_0)$. 
Thus the point $(a_0, c_0)$ lies in (the boundary of) $\mcMhat(x_0, z_0, f_0)$. It is a critical point of the Morse function 
$f_{1
\left[\begin{smallmatrix} x_0 \\ z_0\end{smallmatrix} \right]}|_{\mcMhat(x_0, y_0,f_0) \x \mcMhat(y_0, z_0, f_0)}= 
f_{1\left[\begin{smallmatrix} x_0 \\ y_0\end{smallmatrix} \right]} + f_{1\left[\begin{smallmatrix} y_0 \\ z_0\end{smallmatrix} \right]}$. 

%%%%%%%%%%%%%%%%%%%%%%%%%%%%%%%%%%%%%%%%%%%%%%%%%%%%%%%%%%%%%%%%%%%

\vspace{2mm}

%\subsubsection*{Case $l=2$ and $p=1$}

Now consider the case $l=2$ and $p=1$ and the space 
%\beqs
$X(2)\x_1 X(2)$
%\eeqs
given by 
\begin{gather*}
\left\{\left( (c_2, \mcMhat(y_1, z_1, f_{1 \left[\begin{smallmatrix} \al_0 \\ \be_0 \end{smallmatrix} \right]})), (a_2, \mcMhat(x_1, y_1, f_{1\left[\begin{smallmatrix} \al_0 \\ \be_0 \end{smallmatrix} \right]}))\right)
\left|
\begin{aligned}
& \al_0, \be_0 \in \Crit(f_0), \\
& x_1, y_1, z_1 \in \Crit(f_{1\left[\begin{smallmatrix} \al_0 \\ \be_0 \end{smallmatrix} \right]}), \\
& a_2 \in \Crit(f_{2 \left[ \begin{smallmatrix} x_1, \al_0 \\ y_1, \be_0 \end{smallmatrix} \right]}) ,\\
& c_2 \in \Crit(f_{2 \left[ \begin{smallmatrix} y_1, \al_0 \\ z_1, \be_0 \end{smallmatrix} \right]})
\end{aligned}
\right.
\right\}
\end{gather*}
and define
\beqs
(c_2, \mcMhat(y_1, z_1, f_{1 \left[ \begin{smallmatrix} \al_0 \\ \be_0 \end{smallmatrix} \right]})) \circ_1 (a_2, \mcMhat(x_1, y_1, f_{1 \left[ \begin{smallmatrix} \al_0 \\ \be_0 \end{smallmatrix} \right]})):= \left( (a_2, c_2), \mcMhat(x_1, z_1, f_{1 \left[ \begin{smallmatrix} \al_0 \\ \be_0 \end{smallmatrix} \right]})\right).
\eeqs
Geometrically, we are doing the same as for $X(1) \x_0 X(1)$ except that we are on the space $\mcMhat(\al_0, \be_0, f_0)$ instead of $M$: we glue the Morse trajectories from $x_1$ to $y_1$ (i.e. $\mcMhat(x_1, y_1, f_{1\left[ \begin{smallmatrix} \al_0 \\ \be_0 \end{smallmatrix} \right]})$) with the Morse trajectories from $y_1$ to $z_1$ (i.e. $\mcMhat(y_1, z_1, f_{1\left[ \begin{smallmatrix} \al_0 \\ \be_0 \end{smallmatrix} \right]})$). In terms of $n$-category language, we glue the 2-cells $(a_2, \mcMhat(x_1, y_1, f_{1\left[ \begin{smallmatrix} \al_0 \\ \be_0 \end{smallmatrix} \right]}))$ and $(c_2, \mcMhat(y_1, z_1, f_{1 \left[ \begin{smallmatrix} \al_0 \\ \be_0 \end{smallmatrix} \right]}))$ along the 1-cell $(y_1, \mcMhat(\al_0, \be_0, f_0))$ as visualized below.

\begin{figure}[H] %  h verstaerkt Vorrang des Parameters, aber nicht so stark wie [H]

\begin{center}

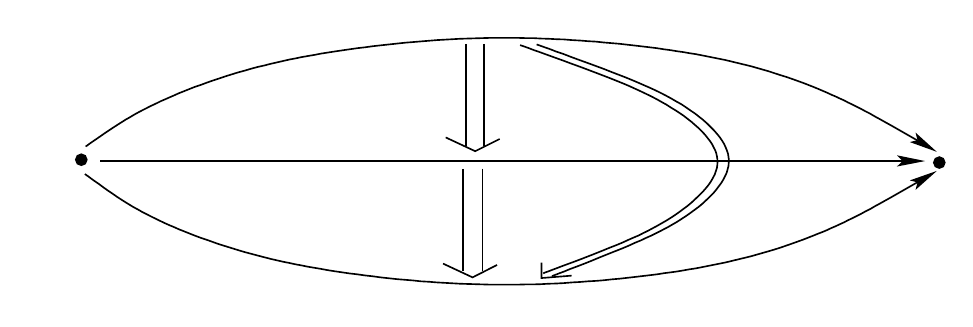
%\caption{Composing of 2-cells along the 1-cell $(y_1, \mcMhat(\al_0, \be_0, f_0))$}

%\label{gluingalong1cell}

\end{center}

\end{figure}

%%%%%%%%%%%%%%%%%%%%%%%%%%%%%%%%%%%%%%%%%%%%%%%%%%%%%%%%%%%%%%%%%%%%%%%%
%%%%%%%%%%%%%%%%%%%  subsection

\subsection{General case: The composite $\circ_p$ of Morse moduli spaces}

After spending some time on motivating the composite for small $l >p \geq 0$, we now define the composite for arbitrary $l > p \geq 0$. To simplify notation, we treat the three cases $p=0$ and $ l-2 \geq p \geq 1$ and $p=l-1$ separately.

\subsubsection*{Case $l \in \N$ and $p=0$}

There are no $\al$'s and $\be$'s such that the `history index' starts with $x_0$, $y_0$, $z_0$. We set
\begin{align*}
 & \left( c_l, \mcMhat(c_{l-1}, d_{l-1}, f_{l-1 \left[ \begin{smallmatrix} c_{l-2}, \dots, c_1, y_0 \\ d_{l-2}, \dots, d_1, z_0 \end{smallmatrix} \right]}) \right) \\
& \quad \circ_0
\left(
a_l, \mcMhat(a_{l-1}, b_{l-1}, f_{l-1 \left[ \begin{smallmatrix} a_{l-2}, \dots, a_1, x_0 \\ b_{l-2}, \dots, b_1, y_0 \end{smallmatrix} \right]})
\right) \\
& :=
\left(
(a_l, c_l), \mcMhat\bigl((a_{l-1}, c_{l-1}), (b_{l-1}, d_{l-1}), f_{l-1 \left[ \begin{smallmatrix} (a_{l-2}, c_{l-2}), \dots, (a_1, c_1), x_0 \\ (b_{l-2}, d_{l-2}), \dots, (b_1, d_1), z_0 \end{smallmatrix} \right]}\bigr)
\right).
\end{align*}

\subsubsection*{Case $l \in \N$ and $l-2 \geq p \geq 1$}

We set
\begin{align*}
& \left(
c_l, \mcMhat(c_{l-1}, d_{l-1}, f_{l-1 \left[ 
\begin{smallmatrix} 
c_{l-2}, \dots, c_{p+1}, y_p, \al_{p-1}, \dots, \al_0 \\
d_{l-2}, \dots, d_{p+1}, z_p, \be_{p-1}, \dots, \be_0
\end{smallmatrix} \right]})
\right) \\
& \quad \circ_p
\left(
a_l, \mcMhat(a_{l-1}, b_{l-1}, f_{l-1 \left[ 
\begin{smallmatrix} 
a_{l-2}, \dots, a_{p+1}, x_p, \al_{p-1}, \dots, \al_0 \\
b_{l-2}, \dots, b_{p+1}, y_p, \be_{p-1}, \dots, \be_0
\end{smallmatrix} \right]})
\right) \\
& := 
\left(
(a_l,c_l), \mcMhat\bigl((a_{l-1}, c_{l-1}), (b_{l-1}, d_{l-1}), f_{l-1 \left[ 
\begin{smallmatrix} 
(a_{l-2}, c_{l-2}), \dots, (a_{p+1}, c_{p+1}), x_p, \al_{p-1}, \dots, \al_0 \\
(b_{l-2}, d_{l-2}), \dots, (b_{p+1}, d_{p+1}), z_p, \be_{p-1}, \dots, \be_0
\end{smallmatrix} \right]}\bigr)
\right) .
\end{align*}
%where
%\beqs
%F:= f_{l-1 \left[ \begin{smallmatrix} \al_0, \dots, \al_{p-1}, x_p, (a_{p+1}, c_{p+1}), \dots, (a_{l-2}, c_{l-2}) \\ \be_0, \dots, \be_{p-1}, z_p, (b_{p+1}, d_{p+1}), \dots, (b_{l-2}, d_{l-2}) \end{smallmatrix} \right]}.
%\eeqs

\subsubsection*{Case $l \in \N$ and $p=l-1$}

There are no $a$'s, $b$'s, $c$'s and $d$'s in the `history index' which ends with $x_{l-1}$, $y_{l-1}$, $z_{l-1}$. We set 
\begin{align*}
& \left(
c_l, \mcMhat(y_{l-1}, z_{l-1}, f_{l-1 \left[ \begin{smallmatrix} \al_{l-2}, \dots, \al_0 \\ \be_{l-2}, \dots, \be_0 \end{smallmatrix} \right]})
\right)  \circ_{l-1}
\left(
a_l, \mcMhat(x_{l-1}, y_{l-1}, f_{l-1 \left[ \begin{smallmatrix} \al_{l-2}, \dots, \al_0 \\ \be_{l-2}, \dots, \be_0 \end{smallmatrix} \right]})
\right) \\
& \qquad := 
\left(
(a_l, c_l), \mcMhat(x_{l-1}, z_{l-1}, f_{l-1 \left[ \begin{smallmatrix} \al_{l-2}, \dots, \al_0 \\ \be_{l-2}, \dots, \be_0 \end{smallmatrix} \right]})
\right).
\end{align*}

%%%%%%%%%%%%%%%%%%%%%%%%%%%%%%%%%%%%%%%%%%%%%%%%%%%%%%%%%%%%%%%%%%%%%%%%%%%%%%%%%%%%%
%%%%%%%%%%%%%%%%%%%%  subsection

\subsection{The $n$-category of Morse moduli spaces}

After defining an $n$-globular set, identity functions and a composite we formulate the main theorem of this paper.

\begin{theorem}
\label{morsencategory}
The $n$-globular set $X=\{X(l) \mid 0\leq l \leq n\}$ together with the above defined identity functions $\mathbf{1}$ and composites $\circ_p$ is an almost strict $n$-category $\mcX$.
\end{theorem}

\begin{proof}
{\em (a) Source and targets of composites:} Let $(C_l, A_l) \in X(l) \x_p X(l)$. Show that, for $p=l-1$, we have $s(C_l \circ_p A_l)= s(A_l)$ and $t(C_l \circ_p A_l)= t(C_l)$.

For $l\geq 1$, we compute 
\begin{align*}
& s(C_l \circ_{l-1} A_l) \\
& = s\left(
\left(c_l, \mcMhat(y_{l-1}, z_{l-1}, f_{
l-1
\left[
\begin{smallmatrix}
\al_{l-2}, \dots, \al_0 \\
\be_{l-2}, \dots, \be_0
\end{smallmatrix}
\right]
})
\right)
\circ_{l-1}
\left(a_l, \mcMhat(x_{l-1}, y_{l-1}, f_{
l-1
\left[ 
\begin{smallmatrix}
\al_{l-2}, \dots, \al_0\\
\be_{l-2}, \dots, \be_0
\end{smallmatrix}
\right]
})
\right)
\right) \\
& = s \left(
(a_l, c_l), \mcMhat(x_{l-1}, z_{l-1}, f_{l-1
\left[
\begin{smallmatrix}
\al_{l-2}, \dots, \al_0 \\
\be_{l-2}, \dots, \be_0
\end{smallmatrix}
\right]
})
\right) \\
& = \left(
x_{l-1}, \mcMhat(\al_{l-2}, \be_{l-2}, f_{l-2
\left[
\begin{smallmatrix}
\al_{l-3}, \dots, \al_0 \\
\be_{l-3}, \dots, \be_0
\end{smallmatrix}
\right]
})
\right) \\
& = s \left(
a_l, \mcMhat(x_{l-1}, y_{l-1}, f_{l-1
\left[
\begin{smallmatrix}
\al_{l-2}, \dots, \al_0 \\
\be_{l-2}, \dots, \be_0
\end{smallmatrix}
\right]
})
\right) \\
& = s(A_l).
\end{align*}
For $l=1$, we find 
\begin{align*}
& s\left((c_1, \mcMhat(y_0, z_0, f_0)) \circ_0 (a_1, \mcMhat(x_0, y_0, f_0)) \right) 
= s \left( (a_1, c_1), \mcMhat(x_0, z_0, f_0) \right) = x_0 \\
& = s \left( a_1, \mcMhat(x_0, y_0, f_0) \right).
\end{align*}
Similar computations yield $t(C_l \circ_{l-1} A_l) = t(C_l)$.

\vspace{2mm}

Furthermore, we have to prove the following. For $(C_l, A_l) \in X(l) \x_p X(l)$, show that, for $0 \leq p \leq l-2$, we have $s(C_l \circ_p A_l)= s(C_l) \circ_ p s(A_l)$ and $t(C_l \circ_p A_l)= t(C_l) \circ_p t(A_l)$.

For $p>0$, we compute
\begin{align*}
& s(C_l \circ_p A_l) \\
& = s \left(
(c_l, \mcMhat(c_{l-1}, d_{l-1}, f_{
l-1
\left[
\begin{smallmatrix}
c_{l-2}, \dots, c_{p+1}, y_p, \al_{p-1}, \dots, \al_0 \\
d_{l-2}, \dots, d_{p+1}, z_p, \be_{p-1}, \dots, \be_0
\end{smallmatrix}
\right]
}) 
\right.\\
& \qquad \qquad 
\left.
\circ_p \
(a_l, \mcMhat(a_{l-1}, b_{l-1}, f_{l-1
\left[
\begin{smallmatrix}
a_{l-2}, \dots, a_{p+1}, x_p, \al_{p-1}, \dots, \al_0 \\
b_{l-2}, \dots, b_{p-1}, y_p, \be_{p-1}, \dots, \be_0
\end{smallmatrix}
\right]
})
\right)\\
& = s \left(
(a_l, c_l), \mcMhat((a_{l-1}, c_{l-1}), (b_{l-1}, d_{l-1}), f_{l-1
\left[
\begin{smallmatrix}
 (a_{l-2}, c_{l-2}), \dots, (a_{p+1}, c_{p+1}), x_p, \al_{p-1}, \dots, \al_0 \\
 (b_{l-2}, d_{l-2}), \dots, (b_{p+1}, d_{p+1}), z_p, \be_{p-1}, \dots, \be_0
\end{smallmatrix}
\right]
})
\right) \\
& = 
\left(
(a_{l-1}, c_{l-1}), \mcMhat((a_{l-2}, c_{l-2}), (b_{l-2}, d_{l-2}), f_{l-2
\left[
\begin{smallmatrix}
 (a_{l-3}, c_{l-3}), \dots, (a_{p+1}, c_{p+1}), x_p, \al_{p-1}, \dots, \al_0 \\
 (b_{l-3}, d_{l-3}), \dots, (b_{p+1}, d_{p+1}), z_p, \be_{p-1}, \dots, \be_0
\end{smallmatrix}
\right]
})
\right) \\
& =
\left(
(c_{l-1}, \mcMhat(c_{l-2}, d_{l-2}, f_{
l-2
\left[
\begin{smallmatrix}
c_{l-3}, \dots, c_{p+1}, y_p, \al_{p-1}, \dots, \al_0 \\
d_{l-3}, \dots, d_{p+1}, z_p, \be_{p-1}, \dots, \be_0
\end{smallmatrix}
\right]
}) 
\right.\\
& \qquad \qquad 
\left.
\circ_p \
(a_{l-1}, \mcMhat(a_{l-2}, b_{l-2}, f_{l-2
\left[
\begin{smallmatrix}
a_{l-3}, \dots, a_{p+1}, x_p, \al_{p-1}, \dots, \al_0 \\
b_{l-3}, \dots, b_{p+1}, y_p, \be_{p-1}, \dots, \be_0
\end{smallmatrix}
\right]
})
\right)\\
& =
s \left(
c_l, \mcMhat(c_{l-1}, d_{l-1}, f_{
l-1
\left[
\begin{smallmatrix}
c_{l-2}, \dots, c_{p+1}, y_p, \al_{p-1}, \dots, \al_0 \\
d_{l-2}, \dots, d_{p+1}, z_p, \be_{p-1}, \dots, \be_0
\end{smallmatrix}
\right]
} 
\right)\\
& \qquad \qquad 
\circ_p \
s \left(
a_l, \mcMhat(a_{l-1}, b_{l-1}, f_{l-1
\left[
\begin{smallmatrix}
a_{l-2}, \dots, a_{p+1}, x_p, \al_{p-1}, \dots, \al_0 \\
b_{l-2}, \dots, b_{p+1}, y_p, \be_{p-1}, \dots, \be_0
\end{smallmatrix}
\right]
}
\right) \\
& = s(C_l) \circ_p s(A_l).
\end{align*}
The case $p=0$ follows similarly. And an analogous computation yields the claim for the target function.

\vspace{2mm}

%%%%%%%%%%%%%%%%%%%%%%%%%%%%%%%%%%%%%%%%%%%%%%%%%%%%%%%%%%%%%%%%%%%%%%%%%%

{\em (b) Sources and targets of identities:} We need to show that $s(\mathbf 1_{A_l})=A_l = t(\mathbf 1_{A_l})$. Letting $A_l= (a_l, \mcMhat(a_{l-1}, b_{l-1}, f_{l-1
\left[
\begin{smallmatrix}
{a_{l-2}, \dots, a_0} \\
{b_{l-2}, \dots, b_0}
\end{smallmatrix}
\right]
})$, we compute
\begin{align*}
s \left(\mathbf 1_{
(a_l, \mcMhat(a_{l-1}, b_{l-1}, f_{l-1
\left[
\begin{smallmatrix}
{a_{l-2}, \dots, a_0} \\
{b_{l-2}, \dots, b_0}
\end{smallmatrix}
\right]
}))}
\right) 
&= s \left(a_l, \mcMhat(a_l, a_l, f_{l
\left[
\begin{smallmatrix}
{a_{l-1}, \dots, a_0} \\
{b_{l-1}, \dots, b_0}
\end{smallmatrix}
\right]
})
\right) \\
&= \left(a_l, \mcMhat(a_{l-1}, b_{l-1}, f_{l-1
\left[
\begin{smallmatrix}
{a_{l-2}, \dots, a_0} \\
{b_{l-2}, \dots, b_0}
\end{smallmatrix}
\right]
}) \right)
\end{align*}
and similar for the target function.

\vspace{2mm}

%%%%%%%%%%%%%%%%%%%%%%%%%%%%%%%%%%%%%%%%%%%%%%%%%%%%%%%%%%%%%%%%%%%%%

{\em (c) Associativity of the composite:} Given $0 \leq p <l \leq n$ and $(E_l, C_l)$, $(C_l, A_l) \in X(l)\x_p X(l)$, we need to prove $(E_l \circ_p C_l) \circ_p A_l = E_l \circ_p ( C_l \circ_p A_l)$.

We set 
\begin{align*}
A_l & = \left( a_l, \mcMhat(a_{l-1}, b_{l-1}, f_{l-1
\left[
\begin{smallmatrix}
a_{l-2}, \dots, a_{p+1}, w_p, \al_{p-1}, \dots, \al_0 \\
b_{l-2}, \dots, b_{p+1}, x_p, \be_{p-1}, \dots, \be_0
\end{smallmatrix}
\right]
}) \right), \\
C_l & = \left( c_l, \mcMhat(c_{l-1}, d_{l-1}, f_{l-1
\left[
\begin{smallmatrix}
c_{l-2}, \dots, c_{p+1}, x_p, \al_{p-1}, \dots, \al_0 \\
d_{l-2}, \dots, d_{p+1}, y_p, \be_{p-1}, \dots, \be_0
\end{smallmatrix}
\right]
}) \right), \\
E_l & = \left( e_l, \mcMhat(e_{l-1}, g_{l-1}, f_{l-1
\left[
\begin{smallmatrix}
e_{l-2}, \dots, e_{p+1}, y_p, \al_{p-1}, \dots, \al_0 \\
g_{l-2}, \dots, g_{p+1}, z_p, \be_{p-1}, \dots, \be_0
\end{smallmatrix}
\right]
}) \right)
\end{align*}
and compute
\begin{align*}
& (E_l \circ_p C_l) \circ_p A_l  \\
& =
\left(
(a_l, (c_l, e_l)), \mcMhat\bigl( (a_{l-1}, (c_{l-1}, e_{l-1})), (b_{l-1}, (d_{l-1}, g_{l-1})),F \bigr)
\right)
\end{align*}
where
\begin{align*}
 F:= f_{l-1
\left[
\begin{smallmatrix}
(a_{l-2}, (c_{l-2}, e_{l-2})), \dots, (a_{p+1}, (c_{p+1}, e_{p+1})), w_p, \al_{p-1}, \dots, \al_0 \\
(b_{l-2}, (d_{l-2}, g_{l-2})), \dots, (b_{p+1}, (d_{p+1}, g_{p+1})), z_p, \be_{p-1}, \dots, \be_0
\end{smallmatrix}
\right]
}.
\end{align*}
On the other hand, we obtain 
\begin{align*}
& E_l \circ_p (C_l \circ_p A_l) \\
&= \left(
((a_l, c_l), e_l), \mcMhat( ((a_{l-1}, c_{l-1}), e_{l-1}), ((b_{l-1}, d_{l-1}), g_{l-1}),\Fbar )
\right)
\end{align*}
where
\begin{align*}
 \Fbar:= f_{l-1
\left[
\begin{smallmatrix}
 ((a_{l-2}, c_{l-2}), e_{l-2}), \dots, ((a_{p+1}, c_{p+1}), e_{p+1}), w_p, \al_{p-1}, \dots, \al_0 \\
 ((b_{l-2}, d_{l-2}), g_{l-2}), \dots, ((b_{p+1}, d_{p+1}), g_{p+1}), z_p, \be_{p-1}, \dots, \be_0
\end{smallmatrix}
\right]
}.
\end{align*}
Geometers usually consider the cartesian product as associative, but if one wants to be rigorous, it is certainly associative up to canonical isomorphism. And the same holds for the gluing of Morse trajectories (cf. \refgluingassociative). Thus, possibly up to canonical isomorphism, $(E_l \circ_p C_l) \circ_p A_l = E_l \circ_p (C_l \circ_p A_l)$.
Note that for $l=1$ and $p=0$, the associativity of the composite reduces to the associativity of the gluing procedure:
\begin{align*}
& (c_1, \mcMhat(c_0, d_0, f_0)) 
\circ_0 \left( 
(b_1, \mcMhat(b_0, c_0, f_0)) 
\circ_0
(a_1, \mcMhat(a_0, b_0, f_0))
\right) \\
& = \left(
((a_1, b_1), c_1), \mcMhat(a_0, d_0, f_0)
\right) \\
& \stackrel{Th. \ref{gluingassociative}}{=}
\left(
(a_1,( b_1, c_1)), \mcMhat(a_0, d_0, f_0)
\right) \\
&= \left(
(c_1, \mcMhat(c_0, d_0, f_0)) 
\circ_0  
(b_1, \mcMhat(b_0, c_0, f_0)) 
\right)
\circ_0
(a_1, \mcMhat(a_0, b_0, f_0)).
\end{align*}

\vspace{2mm}

%%%%%%%%%%%%%%%%%%%%%%%%%%%%%%%%%%%%%%%%%%%%%%%%%%%%%%%%%%%%%%%%%%%%%%%%%%%%%%

{\em (d) Identities:} For $0 \leq p < l \leq n$ and $A_l \in X(l)$, we have to show 
\beqs
\mathbf 1^{l-p}(t^{l-p}(A_l)) \circ_p A_l = A_l = A_l \circ_p \mathbf 1^{l-p}(s^{l-p}(A_l)).
\eeqs

Let 
$A_l = (a_l, \mcMhat(a_{l-1}, b_{l-1}, f_{l-1
\left[
\begin{smallmatrix}
 a_{l-2}, \dots, a_0 \\
 b_{l-2}, \dots , b_0
\end{smallmatrix}
\right]
}))
$ 
and compute 
\begin{align*}
& \mathbf 1^{l-p}\left(t^{l-p}\left(a_l, \mcMhat(a_{l-1}, b_{l-1}, f_{l-1
\left[
\begin{smallmatrix}
 a_{l-2}, \dots, a_0 \\
 b_{l-2}, \dots , b_0
\end{smallmatrix}
\right]
})\right)\right) \\
& \quad = \mathbf 1^{l-p} \left(
b_p, \mcMhat(a_{p-1}, b_{p-1}, f_{p-1
\left[
\begin{smallmatrix}
 a_{p-2}, \dots, a_0 \\
 b_{p-2}, \dots , b_0
\end{smallmatrix}
\right]
})
\right) \\
& \quad  = \left(
b_p, \mcMhat(b_p, b_p, f_{l-1
\left[
\begin{smallmatrix}
b_p, \dots, b_p, a_{p-1}, \dots, a_0 \\
b_p, \dots, b_p, b_{p-1}, \dots, b_0
\end{smallmatrix}
\right]
})
\right) 
\end{align*}
where we identified the critical point $b_p$ with the moduli space $\mcMhat(b_p, b_p)$ and with the critical point $b_{p+1}$ on the moduli space $\mcMhat(b_p, b_p)$ etc. Thus we obtained $l-p-1$ times $b_p$ in each line of the subscript.
Now we compute
\begin{align*}
& \left(
b_p, \mcMhat(b_p, b_p, f_{l-1
\left[
\begin{smallmatrix}
b_p, \dots, b_p, a_{p-1}, \dots, a_0 \\
b_p, \dots, b_p, b_{p-1}, \dots, b_0
\end{smallmatrix}
\right]
})
\right)
\circ_p
\left(a_l, \mcMhat(a_{l-1}, b_{l-1}, f_{l-1
\left[
\begin{smallmatrix}
 a_{l-2}, \dots, a_0 \\
 b_{l-2}, \dots , b_0
\end{smallmatrix}
\right]
})\right) \\
& = 
\left(
(a_l, b_p), \mcMhat((a_{l-1}, b_p), (b_{l-1}, b_p), f_{l-1
\left[
\begin{smallmatrix}
(a_{l-2},b_p), \dots, (a_{p+1}, b_p), a_p, a_{p-1}, \dots, a_0 \\
(b_{l-2}, b_p), \dots, (b_{p+1}, b_p), b_p, b_{p-1}, \dots, b_0
\end{smallmatrix}
\right]
})
\right).
\end{align*}
Since the product of a space with a point can be canonically identified with the space itself we conclude (up to canonical isomorphism)
\begin{align*}
= \left(a_l, \mcMhat(a_{l-1}, b_{l-1}, f_{l-1
\left[
\begin{smallmatrix}
 a_{l-2}, \dots, a_0 \\
 b_{l-2}, \dots , b_0
\end{smallmatrix}
\right]
})\right)
\end{align*}
which yields the claim. The proof for the source function requires the identification of 
\begin{align*}
 \left(
(a_p, a_l), \mcMhat\bigl((a_p, a_{l-1}), (a_p, b_{l-1}), f_{l-1
\left[
\begin{smallmatrix}
(a_p, a_{l-2}) , \dots, (a_p,a_{p+1}), a_p, a_{p-1}, \dots, a_0 \\
(a_p, b_{l-2}), \dots, (a_p, b_{p+1}), b_p, b_{p-1}, \dots, b_0
\end{smallmatrix}
\right]
}\bigr)
\right)
\end{align*}
with 
%\begin{align*}
$\left(a_l, \mcMhat(a_{l-1}, b_{l-1}, f_{l-1
\left[
\begin{smallmatrix}
 a_{l-2}, \dots, a_0 \\
 b_{l-2}, \dots , b_0
\end{smallmatrix}
\right]
})\right).
%\end{align*}
$
\vspace{2mm}

%%%%%%%%%%%%%%%%%%%%%%%%%%%%%%%%%%%%%%%%%%%%%%%%%%%%%%%%%%%

{\em (e) Binary interchange:}
Given $0 \leq q <p <l \leq n$ and $(C_l, A_l)$, $(H_l, E_l) \in X(l)\x_pX(l)$ and $(H_l, C_l)$, $(E_l, A_l) \in X(l)\x_q X(l)$, we need to show $(H_l \circ_p E_l) \circ_q (C_l \circ_p A_l) = (H_l \circ_q C_l) \circ_p (E_l \circ _p A_l)$.

The requirements on $A_l$, $C_l$, $E_l$ and $H_l$ lead to
\begin{align*}
A_l & = 
\left(
a_l, \mcMhat(a_{l-1}, b_{l-1}, f_{l-1
\left[
\begin{smallmatrix}
a_{l-2}, \dots, a_{p+1}, x_p, \al_{p-1}, \dots, \al_q, \dots, \al_0 \\
b_{l-2}, \dots, b_{p+1}, y_p, \be_{p-1}, \dots, \be_q, \dots, \be_0
\end{smallmatrix}
\right]
})
\right), \\
C_l & = 
\left(
c_l, \mcMhat(c_{l-1}, d_{l-1}, f_{l-1
\left[
\begin{smallmatrix}
c_{l-2}, \dots, c_{p+1}, y_p, \al_{p-1}, \dots, \al_q, \dots, \al_0 \\
d_{l-2}, \dots, d_{p+1}, z_p, \be_{p-1}, \dots, \be_q, \dots, \be_0
\end{smallmatrix}
\right]
})
\right), \\
E_l & = 
\left(
e_l, \mcMhat\bigl(e_{l-1}, g_{l-1}, f_{l-1
\left[
\begin{smallmatrix}
e_{l-2}, \dots, e_{p+1}, \xbar_p, \ep_{p-1}, \dots, \ep_{q+1}, \be_q, \al_{q-1}, \dots, \al_0 \\
g_{l-2}, \dots, g_{p+1}, \ybar_p, \ga_{p-1}, \dots, \ga_{q+1}, \ga_q, \be_{q-1}, \dots, \be_0
\end{smallmatrix}
\right]
}\bigr)
\right), \\
H_l & = 
\left(
h_l, \mcMhat(h_{l-1}, i_{l-1}, f_{l-1
\left[
\begin{smallmatrix}
h_{l-2}, \dots, h_{p+1}, \ybar_p, \ep_{p-1}, \dots, \ep_{q+1}, \be_q, \al_{q-1}, \dots, \al_0 \\
i_{l-2}, \dots, i_{p+1}, \zbar_p, \ga_{p-1}, \dots, \ga_{q+1}, \ga_q, \be_{q-1}, \dots, \be_0
\end{smallmatrix}
\right]
})
\right).
\end{align*}
We compute
\begin{align*}
& (H_l \circ_p E_l) \circ_q (C_l \circ_p A_l) \\
& = 
\left(
(e_l, h_l), \mcMhat((e_{l-1}, h_{l-1}), (g_{l-1}, i_{l-1}), f_{l-1, \triangle^1}
)	
\right) \\
& \qquad \circ_q
\left(
(a_l, c_l), \mcMhat((a_{l-1}, c_{l-1}), (b_{l-1}, d_{l-1}), f_{l-1, \triangle^2}
)	
\right) \\
& = \left(
[(a_l, c_l), (e_l, h_l)], 
\mcMhat([(a_{l-1}, c_{l-1}), (e_{l-1}, h_{l-1})],[(b_{l-1}, d_{l-1}), (g_{l-1}, i_{l-1})], f_{l-1, \triangle^3})
\right)
\end{align*}
where
\begin{align*}
\triangle^1 
& :=
\left[
\begin{smallmatrix}
(e_{l-2}, h_{l-2}), \dots,(e_{p+1}, h_{p+1}), \xbar_p, \ep_{p-1}, \dots, \ep_{q+1}, \be_q, \al_{q-1}, \dots, \al_0 \\ 
(g_{l-2}, i_{l-2}), \dots,  (g_{p+1}, i_{p+1}), \zbar_p, \ga_{p-1}, \dots, \ga_{q+1}, \ga_q, \be_{q-1}, \dots, \be_0
\end{smallmatrix}
\right],\\ 
\triangle^2 
& :=
\left[
\begin{smallmatrix}
(a_{l-2}, c_{l-2}), \dots, (a_{p+1}, c_{p+1}), x_p, \al_{p-1}, \dots, \al_q, \dots, \al_0 \\
(b_{l-2}, d_{l-2}), \dots, (b_{p+1}, d_{p+1}), z_p, \be_{p-1}, \dots, \be_q, \dots, \be_0
\end{smallmatrix}
\right]
\end{align*}
and $\triangle^3$ is given by
\begin{align*}
\left[
\begin{smallmatrix}
[(a_{l-2}, c_{l-2}), (e_{l-2}, h_{l-2})], \dots,  [(a_{p+1}, c_{p+1}), (e_{p+1}, h_{p+1})],  (x_p, \xbar_p),(\al_{p-1}, \ep_{p-1}), \dots, (\al_{q+1}, \ep_{q+1}), \al_q, \al_{q-1}, \dots, \al_0 \\
[(b_{l-2}, d_{l-2}), (g_{l-2}, i_{l-2})], \dots, [(b_{p+1}, d_{p+1}), (g_{p+1}, i_{p+1})],(z_p, \zbar_p),(\be_{p-1}, \ga_{p-1}), \dots, (\be_{q+1}, \ga_{q+1}), \ga_q, \be_{q-1}, \dots, \be_0
\end{smallmatrix}
\right].
\end{align*}
On the other hand, we calculate
\begin{align*}
& (H_l \circ_q C_l) \circ_p (E_l \circ_p A_l) \\
& = 
\left(
(c_l, h_l), \mcMhat((c_{l-1}, h_{l-1}), (d_{l-1}, i_{l-1}), f_{l-1, \triangle^4}
)	
\right) \\
& \qquad \circ_q
\left(
(a_l, e_l), \mcMhat((a_{l-1}, e_{l-1}), (b_{l-1}, g_{l-1}), f_{l-1, \triangle^5}
)	
\right) \\
& = \left(
[(a_l, e_l), (c_l, h_l)], 
\mcMhat([(a_{l-1}, e_{l-1}), (c_{l-1}, h_{l-1})],[(b_{l-1}, g_{l-1}), (d_{l-1}, i_{l-1})], f_{l-1, \triangle^6})
\right)
\end{align*}
where
\begin{align*}
\triangle^4
& :=
\left[
\begin{smallmatrix}
(c_{l-2}, h_{l-2}), \dots, (c_{p+1}, h_{p+1}),(y_p, \ybar_p), (\al_{p-1}, \ep_{p-1}), \dots,  (\al_{q+1}, \ep_{q+1}), \al_q, \al_{q-1}, \dots, \al_0 \\
(d_{l-2}, i_{l-2}), \dots, (d_{p+1}, i_{p+1}), (z_p, \zbar_p),(\be_{p-1}, \ga_{p-1}), \dots, (\be_{q+1}, \ga_{q+1}), \ga_q, \be_{q-1}, \dots, \be_0
\end{smallmatrix}
\right],\\ 
\triangle^5
& :=
\left[
\begin{smallmatrix}
(a_{l-2}, e_{l-2}), \dots, (a_{p+1}, e_{p+1}),(x_p, \xbar_p),(\al_{p-1}, \ep_{p-1}), \dots,(\al_{q+1}, \ep_{q+1}), \al_q, \al_{q-1}, \dots, \al_0 \\
(b_{l-2}, g_{l-2}), \dots, (b_{p+1}, g_{p+1}),(y_p, \ybar_p),(\be_{p-1}, \ga_{p-1}), \dots,(\be_{q+1}, \ga_{q+1}), \ga_q, \be_{q-1}, \dots, \be_0 
\end{smallmatrix}
\right]
\end{align*}
and $\triangle^6$ is given by
\begin{align*}
\left[
\begin{smallmatrix}
[(a_{l-2}, e_{l-2}), (c_{l-2}, h_{l-2})], \dots,  [(a_{p+1}, e_{p+1}), (c_{p+1}, h_{p+1})],(x_p, \xbar_p), (\al_{p-1}, \ep_{p-1}), \dots, (\al_{q+1}, \ep_{q+1}), \al_q,\al_{q-1}, \dots,\al_0 \\
[(b_{l-2}, g_{l-2}), (d_{l-2}, i_{l-2})], \dots, [(b_{p+1}, g_{p+1}), (d_{p+1}, i_{p+1})], (z_p, \zbar_p), (\be_{p-1}, \ga_{p-1}), \dots,   (\be_{q+1}, \ga_{q+1}), \ga_q, \be_{q-1}, \dots, \be_0
\end{smallmatrix}
\right].
\end{align*}
$\triangle^3$ and $\triangle^6$ share the first half
\begin{align*}
\left[
\begin{smallmatrix}
(x_p, \xbar_p), (\al_{p-1}, \ep_{p-1}), \dots, (\al_{q+1}, \ep_{q+1}), \al_q,\al_{q-1}, \dots,\al_0 \\
(z_p, \zbar_p), (\be_{p-1}, \ga_{p-1}), \dots,   (\be_{q+1}, \ga_{q+1}), \ga_q, \be_{q-1}, \dots, \be_0
\end{smallmatrix}
\right]
\end{align*}
and differ in the second half only up to exchange of the second and third coordinate in the 4-tuples. Thus, up to canonical isomorphism, we obtain the claim.

%%%%%%%%%%%%%%%%%%%%%%%%%%%%%%%%%%%%%%%%%%%%%%%%%%%%%%%%%%%%%%%%%%%%%%%%%%%%%%%%%%%%%%%%%%%%%%%%%

\vspace{2mm}

{\em (f) Nullary interchange:} For $0 \leq p <l <n$ and $(C_l, A_l) \in X(l) \x _p X(l)$, we need to show $\mathbf 1_{C_l} \circ_p \mathbf 1_{A_l} = \mathbf 1_{C_l \circ_p A_l}$.

Let
\begin{align*}
A_l & = \left(a_l, \mcMhat(a_{l-1}, b_{l-1}, f_{l-1
\left[
\begin{smallmatrix}
a_{l-2}, \dots, a_{p+1}, x_p, \al_{p-1}, \dots, \al_0 \\
b_{l-2}, \dots, b_{p+1}, y_p, \be_{p-1}, \dots, \be_0
\end{smallmatrix}
\right]
})
\right), \\
C_l & = \left(c_l, \mcMhat(c_{l-1}, d_{l-1}, f_{l-1
\left[
\begin{smallmatrix}
c_{l-2}, \dots, c_{p+1}, y_p, \al_{p-1}, \dots, \al_0 \\
d_{l-2}, \dots, d_{p+1}, z_p, \be_{p-1}, \dots, \be_0
\end{smallmatrix}
\right]
})
\right)
\end{align*}
and compute
\begin{align*}
& \mathbf 1_{C_l} \circ_p \mathbf 1_{A_l} \\
& = \left(c_l, \mcMhat(c_l, c_l, f_{l
\left[
\begin{smallmatrix}
c_{l-1}, \dots, c_{p+1}, y_p, \al_{p-1}, \dots, \al_0 \\
d_{l-1}, \dots, d_{p+1}, z_p, \be_{p-1}, \dots, \be_0
\end{smallmatrix}
\right]
})
\right) \\
& \qquad  \circ_p
\left(a_l, \mcMhat(a_l, a_l, f_{l
\left[
\begin{smallmatrix}
a_{l-1}, \dots, a_{p+1}, x_p, \al_{p-1}, \dots, \al_0 \\
b_{l-1}, \dots, b_{p+1}, y_p, \be_{p-1}, \dots, \be_0
\end{smallmatrix}
\right]
})
\right) \\
& = \left((a_l, c_l), \mcMhat((a_l, c_l), (a_l, c_l), f_{l
\left[
\begin{smallmatrix}
(a_{l-1}, c_{l-1}), \dots,  (a_{p+1}, c_{p+1}), x_p, \al_{p-1}, \dots, \al_0 \\
 (b_{l-1}, d_{l-1}), \dots, (b_{p+1},d_{q+1}), z_p, \be_{p-1}, \dots, \be_0
\end{smallmatrix}
\right]
})
\right) \\
&=
\mathbf 1 \left((a_l, c_l), \mcMhat((a_{l-1}, c_{l-1}), (b_{l-1}, d_{l-1}), f_{l-1
\left[
\begin{smallmatrix}
(a_{l-2}, c_{l-2}), \dots,  (a_{p+1}, c_{p+1}), x_p, \al_{p-1}, \dots, \al_0 \\
 (b_{l-2}, d_{l-2}), \dots, (b_{p+1},d_{q+1}), z_p, \be_{p-1}, \dots, \be_0
\end{smallmatrix}
\right]
})
\right) \\
&= \mathbf 1_{C_l \circ_p A_l} .
\end{align*}
which finishes the proof of \refmorsencategory.
\end{proof}

%%%%%%%%%%%%%%%%%%%%%%%%%%%%%%%%%%%%%%%%%%%%%%%%%%%%%%%%%%%%%%%%%%%%%%%%%%%%%%%%%%%%%%%%%%%%%%%%
%%%%%%%%%%%%%%%%%%%%%%%%%%%%%%%%%%%%%%%%%%%%%%%%%%%%%%%%%%%%%%%%%%%%%%%%%%%%%%%%%%%%%%%%%%%%
%%%%%%%%%%%  new section  %%%%%%%%%%%%%%%%%%%%%%%%%%%%%%%%%%%%%%%%%%%%%%%%%%%
%%%%%%%%%%  new section  %%%%%%%%%%%%%%%%%%%%%%%%%%%%%%%%%%%%%%%%%%%%%%%%%%%%%%%%
%%%%%%%%%%%%%%%%%%%%%%%%%%%%%%%%%%%%%%%%%%%%%%%%%%%%%%%%%%%%%%%%%%%%%%%%%%%%%

\section{Examples}

%%%%%%%%%%%%%%%%%%%%%%%%%%%%%%%%%%%%%%%%%%%%%%%%%%%%%%%%%%%%%%%%%%%%%%%%%%%%%%
%%%%%%%%%%% subsection %%%%%%%%%%%%%%%%%

\subsection{The standard $n$-sphere}

Consider the $n$-dimensional sphere $\mathbb S^n:=\{(p_1 , \dots , p_{n+1}) \in \R^{n+1} \mid p_0^2 + \cdots + p_n^2=1\}$ with the height function $f_0\colon \mathbb S^n \to \R$, $f_0(p_1, \dots, p_{n+1}):=p_{n+1}$ as Morse function and use the induced metric from $\R^{n+1}$. It has two critical points $x_0$ and $y_0$, namely the north and the south pole, with $\Ind(x_0)=n$ and $\Ind(y_0)=0$. 
Thus we have 
\beqs
X(0)=\{x_0, y_0\}.
\eeqs
The moduli space $\mcMhat(x_0, y_0, f_0)$ can be identified with $\mathbb S^{n-1}$ which has no boundary. Thus there are no lower dimensional boundary strata which could impose compatibility conditions on the chosen Morse function on the Morse moduli space. Let $f_1=f_{1 \left[ \begin{smallmatrix} x_0 \\ y_0 \end{smallmatrix} \right]}$ be the height function on $\mathbb S^{n-1}\simeq \mcMhat(x_0, y_0, f_0)$ with critical points $x_1$ (north pole) and $y_1$ (south pole). 
We obtain
\beqs
X(1)=\{(x_1, \mcMhat(x_0, y_0, f_0)), (y_1, \mcMhat(x_0, y_0, f_0))\}.
\eeqs
The moduli space $\mcMhat(x_1, y_1, f_1)$ can be identified with $\mathbb S^{n-2}$ and we choose as Morse function $f_2=f_{2 \left[ \begin{smallmatrix} x_1 , x_0 \\ y_1, y_0 \end{smallmatrix} \right]}$ the height function on $\mathbb S^{n-2}$. We get
\beqs
X(2)=\{(x_2, \mcMhat(x_1, y_1, f_1)), (y_2, \mcMhat(x_1, y_1, f_1))\}
\eeqs
where $x_2$ is the north pole and $y_2$ the south pole.
Iterating this procedure, we find eventually
\beqs
X(n)=\{(x_n, \mcM(x_{n-1}, y_{n-1}, f_{n-1})), (y_n, \mcMhat(x_{n-1}, y_{n-1}, f_{n-1}))\}
\eeqs
where $\mcMhat(x_{n-1}, y_{n-1}, f_{n-1})$ can be identified with $\mathbb S^0$ which again can be identified with the critical points $\{x_n\} \cup \{y_n\}$.
The process terminates with
\beqs
X(n+1)=\{(x_n, \mcMhat(x_n, x_n, f_n)), (y_n, \mcMhat(y_n, y_n, f_n))\}
\eeqs
where the `point' $x_n$ can be identified with the `space' $\mcMhat(x_n, x_n, f_n)$ and similar for $y_{n}$.

\vspace{2mm}

Now we want to look for the possible composites.
For sake of readability, we only consider the {\bf case} ${\mathbf n \mathbf = \mathbf 2}$. The general case goes analogously. Moreover, to simplify notation, we drop the Morse function in the moduli spaces. We compute
\begin{align*}
t(x_1, \mcMhat(x_0, y_0))& =y_0,  & s(x_1, \mcMhat(x_0, y_0))=x_0,  \\
t(y_1, \mcMhat(x_0, y_0))& =y_0,  & s(y_1, \mcMhat(x_0, y_0))=x_0.
\end{align*}
and conclude
\beqs
X(1) \x_0 X(1)= \emptyset.
\eeqs
And computing
\begin{align*}
s^2(x_2, \mcMhat(x_1, y_1)) &=s(x_1, \mcMhat(x_0, y_0)) =x_0, \\
t^2(x_2, \mcMhat(x_1, y_1)) & =s(y_1, \mcMhat(x_0, y_0)) =y_0, \\
s^2(y_2, \mcMhat(x_1, y_1)) &=s(x_1, \mcMhat(x_0, y_0))  =x_0, \\
t^2(y_2, \mcMhat(x_1, y_1)) & =s(y_1, \mcMhat(x_0, y_0)) =y_0
\end{align*}
yields
\beqs
X(2)\x_0 X(2)= \emptyset= X(2) \x_1 X(2).
\eeqs
Calculating
\begin{align*}
s^3(x_2, \mcMhat(x_2, x_2))& = & s^2(x_2, \mcMhat(x_1, y_1))& = & s (x_1, \mcMhat(x_0, y_0)) &=& x_0 , \\
t^3(x_2, \mcMhat(x_2, x_2))& = & t^2(x_2, \mcMhat(x_1, y_1)) & = & t(y_1, \mcMhat(x_0, y_0)) & = & y_0, \\
s^3(y_2, \mcMhat(y_2, y_2))& = & s^2(y_2, \mcMhat(x_1, y_1))& = & s (x_1, \mcMhat(x_0, y_0)) &=& x_0 , \\
t^3(y_2, \mcMhat(y_2, y_2))& = & t^2(y_2, \mcMhat(x_1, y_1)) & = & t(y_1, \mcMhat(x_0, y_0)) & = & y_0
\end{align*}
leads to
\beqs
X(3)\x_lX(3)=\emptyset \quad \mbox{for } l=0, 1, 2.
\eeqs
Geometrically the lack of composites is due to the fact that there are only two critical points on each level such that there is no gluing or breaking of Morse trajectories.

%%%%%%%%%%%%%%%%%%%%%%%%%%%%%%%%%%%%%%%%%%%%%%%%%%%%%%%%%%%%%%%%%%%%%%%%%%%%%%%%%%
%%%%%%%%%%%%%%%%%%%%%5  subsection  %%%%%%%%%%%%%%%%%%%%%%%%%%%%%%%%%%%%%%%%%%%

\subsection{The deformed $2$-sphere}

Let $M$ be the deformed 2-sphere 

\begin{figure}[h] %  verstaerkt Vorrang des Parameters, aber nicht so stark wie [H]

\begin{center}

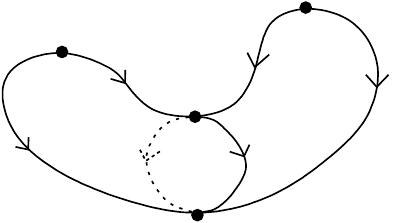
%\caption{Morse trajectories of the height function on a deformed sphere}
\label{sphere}

\end{center}

\end{figure}

%sketched in Figure \ref{sphere}. 
Choose the induced metric from $\R^3$ (suitably adjusted near the critical points) and take the height function, denoted by $f_0$, as a Morse function. The Morse trajectories are the negative gradient flow lines. For sake of readability, we drop the Morse function in the notion of the moduli spaces.
We have four critical points $\Crit(f_0)=\{w, x, y, z\}$ with Morse index $\Ind(w)=0$, $\Ind(x)=2$, $\Ind(y)=1$ and $\Ind(z)=2$. For the moduli spaces holds $\dim \mcMhat(x,w)= \dim\mcMhat(z,w)=1$ and $\dim \mcMhat(y,w)=\dim \mcMhat(x,y) = \dim \mcMhat(z,y)=0$ with cardinality $\#\mcMhat(y,w)=2$, $\#\mcMhat(x,y) =1$ and $\# \mcMhat(z,y)=1$. All other moduli spaces vanish. $\mcMhat(y,w)$ has two connected components which we denote by $\mcMhat(y,w) = \mcMhat(y,w)_a \cup \mcMhat(y,w)_b$.

We have $X(0)=\Crit(f_0)=\{w, x, y, z\}$. 
$\mcMhat(x,w)$ is an interval whose boundary is given by 
\beqs
\mcMhat(x,y) \x \mcMhat(y,w) = \{ (\mcMhat(x,y), \mcMhat(y,w)_a)\} \cup \{(\mcMhat(x,y), \mcMhat(y,w)_b)\} 
\eeqs
and similar for $\mcMhat(z,w)$. If we consider the (components of one of the) {\em zero} dimensional moduli spaces as {\em points} instead of {\em spaces}, we write $\mfmhat(\dots)$ instead of $\mcMhat(\dots)$. Now define the Morse function $f_1$ as follows. Assume $f_{1\left[ \begin{smallmatrix} x \\ w \end{smallmatrix} \right]}$ on $\mcMhat(x,w, f_0)$ and $f_{1\left[ \begin{smallmatrix} z \\ w \end{smallmatrix} \right]}$ on $\mcMhat(z,w,f_0)$ to be strictly monotone with the same (positive) maximum at $\left(\mfmhat(x,y), \mfmhat(y,w)_a\right)$ and the same (positive) minimum at $(\mfmhat(x,y), \mfmhat(y,w)_b)$ which are the only critical points.
With this notion, we find
\beqs
X(1)=
\left\{
\begin{gathered}
\left(\mfmhat(y,w)_a, \mcMhat(y,w)\right), \left( \mfmhat(y,w)_b, \mcMhat(y,w)\right), \\ 
\left(\mfmhat(x,y), \mcMhat(x,y)\right), \left(\mfmhat(z,y), \mcMhat(z,y)\right), \\
\left(\left(\mfmhat(x,y), \mfmhat(y,w)_a\right), \mcMhat(x,w)\right), 
\left(\left(\mfmhat(x,y), \mfmhat(y,w)_b\right), \mcMhat(x,w)\right), \\
\left(\left(\mfmhat(z,y), \mfmhat(y,w)_a\right), \mcMhat(z,w)\right), 
\left(\left(\mfmhat(z,y), \mfmhat(y,w)_b\right), \mcMhat(z,w)\right)
\end{gathered}
\right\}.
\eeqs
We compute 
\begin{align*}
X(1) \x_0 X(1)& = \{(\xiti, \xi) \in X(1) \x X(1) \mid s(\xiti) = t(\xi)\} \\
& = \left\{
\begin{gathered}
\left((\mfmhat(y,w)_a, \mcMhat(y,w)), (\mfmhat(x,y), \mcMhat(x,y))\right), \\
\left((\mfmhat(y,w)_b, \mcMhat(y,w)), (\mfmhat(x,y), \mcMhat(x,y))\right) ,\\
\left((\mfmhat(y,w)_a, \mcMhat(y,w)), (\mfmhat(z,y), \mcMhat(z,y))\right), \\
\left((\mfmhat(y,w)_b, \mcMhat(y,w)), (\mfmhat(z,y), \mcMhat(z,y))\right)
\end{gathered}
\right\}.
\end{align*}
%Let us compute the composite of
%\beqs
%\left((\mfmhat(y,w)_a, \mcMhat(y,w)), (\mfmhat(x,y), \mcMhat(x,y))\right) \in X(1)\x_0 X(1).
%\eeqs
%We obtain
We compute the composite
\begin{align*}
\left(\mfmhat(x,y), \mcMhat(x,y)\right) \circ_0 \left(\mfmhat(y,w)_a, \mcMhat(y,w)\right) 
 = \left( ( \mfmhat(x,y),\mfmhat(y,w)_a), \mcMhat(x,w)  \right).
\end{align*}
The other elements of $X(1) \x_0 X(1)$ work similarly. Geometrically we are gluing Morse trajectories.
Now abbreviate 
\begin{align*}
P & :=\mcMhat\left((\mfmhat(x,y), \mfmhat(y,w)_a), (\mfmhat(x,y), \mfmhat(y,w)_b)\right), \\
p & :=\mfmhat\left((\mfmhat(x,y), \mfmhat(y,w)_a), (\mfmhat(x,y), \mfmhat(y,w)_b)\right),\\
Q & := \mcMhat\left((\mfmhat(z,y), \mfmhat(y,w)_a), (\mfmhat(z,y), \mfmhat(y,w)_b)\right), \\
q & := \mfmhat\left((\mfmhat(z,y), \mfmhat(y,w)_a), (\mfmhat(z,y), \mfmhat(y,w)_b)\right)
\end{align*}
and we obtain
\beqs
X(2)=
\left\{
\begin{gathered}
\left. 
\begin{gathered}
\left(\mfmhat(y,w)_i, \mcMhat(\mfmhat(y,w)_i, \mfmhat(y,w)_i )\right)\\ 
\left(\mfmhat(r,y), \mcMhat(\mfmhat(r,y), \mfmhat(r,y) )\right), \\ 
(p,P), (q, Q)
\end{gathered}
\right| 
i \in \{a,b\}, r \in \{x,z\}
\end{gathered}
\right\}.
\eeqs
The appearing moduli spaces are singletons such that $X(l)$ for $l \geq 3$ will only contain `trivial' elements of the form $(\xi, \mcMhat(\xi, \xi))$.
We compute
\begin{align*}
& s^2 \left(\mfmhat(y,w)_i, \mcMhat(\mfmhat(y,w)_i, \mfmhat(y,w)_i )\right)
= s \left(\mfmhat(y,w)_i, \mcMhat(y,w  )  \right) 
= y, \\
& t^2 \left(\mfmhat(y,w)_i, \mcMhat(\mfmhat(y,w)_i, \mfmhat(y,w)_i )\right)
= t \left( \mfmhat(y,w)_i, \mcMhat(y,w)   \right) 
= w, \\
& s^2 \left(\mfmhat(r,y), \mcMhat(\mfmhat(r,y), \mfmhat(r,y) )\right) 
= s \left( \mfmhat(r,y), \mcMhat(r, y) \right)
= r \in \{x,z\}, \\
& t^2 \left(\mfmhat(r,y), \mcMhat(\mfmhat(r,y), \mfmhat(r,y) )\right) 
= t \left( \mfmhat(r,y), \mcMhat(r, y) \right)
= y , \\
& s^2(p,P)= s\left( (\mfmhat(x,y), \mfmhat(y,w)_a), \mcMhat(x,w) \right)= x ,\\
& t^2(p,P) = t\left( (\mfmhat(x,y), \mfmhat(y,w)_b), \mcMhat(x,w) \right)=w , \\
& s^2(q,Q)= s\left( (\mfmhat(z,y), \mfmhat(y,w)_a), \mcMhat(z,w) \right)= z ,\\
& t^2(q,Q) = t\left( (\mfmhat(z,y), \mfmhat(y,w)_b), \mcMhat(z,w) \right)=w 
\end{align*}
which implies
\begin{align*}
& X(2) \x_1 X(2) \\
& \quad  = \{(\xiti, \xi) \in X(2) \x X(2) \mid t(\xi)=s(\xiti)\} \\
& \quad =  
\left\{
\begin{gathered}
\left(
\left(\mfmhat(y,w)_i, \mcMhat(\mfmhat(y,w)_i, \mfmhat(y,w)_i )\right), 
\left(\mfmhat(y,w)_i, \mcMhat(\mfmhat(y,w)_i, \mfmhat(y,w)_i )\right)
\right), \\
\left(
\left(\mfmhat(r,y), \mcMhat(\mfmhat(r,y), \mfmhat(r,y) )\right) ,
\left(\mfmhat(r,y), \mcMhat(\mfmhat(r,y), \mfmhat(r,y) )\right) 
\right) \\
\mbox{for } i \in \{a,b\}, r \in \{x,z\}
\end{gathered}
\right\}
\end{align*}
and we compute for instance
\begin{align*}
& \left(\mfmhat(r,y), \mcMhat(\mfmhat(r,y), \mfmhat(r,y) )\right) \circ_1  \left(\mfmhat(r,y), \mcMhat(\mfmhat(r,y), \mfmhat(r,y) )\right) \\
& \quad = 
\left(
(\mfmhat(r,y),\mfmhat(r,y)), \mcMhat((\mfmhat(r,y),\mfmhat(r,y)), (\mfmhat(r,y),\mfmhat(r,y))  )
\right) \\
& \quad \simeq \left(\mfmhat(r,y), \mcMhat(\mfmhat(r,y), \mfmhat(r,y) )\right)
\end{align*}
which is due to the fact that we are working on a space which consists of a single point.
Moreover we have
\begin{align*}
& X(2) \x_0 X(2) \\
& \quad = \{(\xiti, \xi) \in X(2) \x X(2) \mid t^2(\xi)=s^2(\xiti) \} \\
& \quad =
\left\{
\begin{gathered}
\left(
(\mfmhat(y,w)_i, \mcMhat(\mfmhat(y,w)_i, \mfmhat(y,w)_i ), 
(\mfmhat(r,y), \mcMhat(\mfmhat(r,y), \mfmhat(r,y) )) 
\right) \\
 \mbox{for }i \in \{a,b\}, r \in \{x,z\}
\end{gathered}
\right\}
\end{align*}
and we compute
\begin{align*}
& \left(\mfmhat(y,w)_i, \mcMhat(\mfmhat(y,w)_i, \mfmhat(y,w)_i ) \right) \circ_0
\left(\mfmhat(r,y), \mcMhat(\mfmhat(r,y), \mfmhat(r,y) )\right) \\
& \quad = 
\left((\mfmhat(r,y), \mfmhat(y,w)_i), \mcMhat( ( \mfmhat(r,y)  ,\mfmhat(y,w)_i) ,( \mfmhat(r,y), \mfmhat(y,w)_i   ) \right).
\end{align*}

%%%%%%%%%%%%%%%%%%%%%%%%%%%%%%%%%%%%%%%%%%%%%%%%%%%%%%%%%%%%%%%%%%%%%%%%%%%%%%%%%%%%%%%%%%%%%%%%%%%%%%%%%%%%%%%%
%%%%%%%%%%%%%%%  bibliography  %%%%%%%%%%%%%%%%%%%%%%%%%%%%%%%%%%%%%%%%%%%%%%%%%%%%%%%%%%%%%%%%%%%%%%%%%%%%%
%%%%%%%%%%%%%%%%%%%%%%%%%%%%%%%%%%%%%%%%%%%%%%%%%%%%%%%%%%%%%%%%%%%%%%%%%%%%%%%%%%%%%%%%%%%%%%%%%%%%%%%%%%%

\end{document}